\documentclass[12pt,reqno]{amsart}
\usepackage{hyperref}
\hypersetup{colorlinks=true, linkcolor=blue, citecolor=red, urlcolor=blue}

\usepackage{amssymb,amsmath,amsfonts,latexsym}
\usepackage{ragged2e}
\usepackage{bm}
\usepackage{mathtools}
\usepackage{enumerate}
\usepackage{array, graphics,xcolor}
\allowdisplaybreaks
\usepackage{mathrsfs}
\usepackage{mathtools}
\usepackage{blkarray, bigstrut}

\numberwithin{equation}{section}
\newtheorem{dfn}{Definition}[section]
\newtheorem{thm}[dfn]{Theorem}
\newtheorem{lma}[dfn]{Lemma}
\newtheorem{prop}[dfn]{Proposition}

\newtheorem{crlre}[dfn]{Corollary}

\newtheorem{Not}[dfn]{Notations}

\DeclarePairedDelimiterX{\norm}[1]{\lVert}{\rVert}{#1}
\DeclarePairedDelimiterX{\bnorm}[1]{\big\lVert}{\big\rVert}{#1}
\DeclarePairedDelimiterX{\Bnorm}[1]{\Big\lVert}{\Big\rVert}{#1}

\newcommand{\R}{\mathbb{R}}
\newcommand{\Z}{\mathbb{Z}}
\newcommand{\N}{\mathbb{N}}

\newcommand{\Nat}{\mathbb{N}}

\newcommand{\D}{\mathbb{D}}

\newcommand{\hil}{\mathcal{H}}

\newcommand{\boh}{\mathcal{B}_1(\hil)}

\newcommand{\cir}{\mathbb{T}}
\newcommand{\T}{\mathbb{T}}

\newcommand{\la}{\langle}
\newcommand{\ra}{\rangle}

\newcommand{\bnh}{\mathcal{B}_n(\hil)}

\newcommand{\Tr}{\operatorname{Tr}}

\newcommand{\mult}{\textit{Mult}}
\newcommand{\lin}{\textit{Lin}}

\newcommand{\Fcal}{\mathcal{F}}

\newcommand{\Fpt}[1]{\mathfrak{F}_{#1}(\cir)}
\newcommand{\pt}[1]{\mathfrak{F}_{#1}}
\newcommand{\tFpt}[1]{\widetilde{\mathfrak{F}}_{#1}(\cir)}
\newcommand{\FptR}[1]{\mathfrak{F}_{#1}(\R)}

\def\Im{{\mathrm{Im}\,}}
\def\supp{{\mathrm{supp}\,}}

\newcommand{\dm}{\tau}

\begin{document}
	
	%\title{Higher-Order Trace Formulas for Unitary and Comparable Resolvent Perturbations II}
	\title{ Higher-Order Trace Formulas for Contractive and Dissipative Operators}
	
	\author[Chattopadhyay] {Arup Chattopadhyay}
	\address{Department of Mathematics, Indian Institute of Technology Guwahati, Guwahati, 781039, India}
	\email{arupchatt@iitg.ac.in, 2003arupchattopadhyay@gmail.com}
	
	\author[Pradhan]{Chandan Pradhan}
	\address{ Department of Mathematics, Indian Institute of Science, Bangalore, 560012, India}
	\email{chandan.pradhan2108@gmail.com, chandanp@iisc.ac.in}
	
	\author[Skripka] {Anna Skripka}
	\address{Department of Mathematics and Statistics, University of New Mexico, 311 Terrace Street NE, Albuquerque, NM 87106, USA}
	\email{skripka@math.unm.edu}

	
	\begin{abstract}
		We establish higher order trace formulas for pairs of contractions along a multiplicative path generated by a self-adjoint operator in a Schatten-von Neumann ideal, removing earlier stringent restrictions on the kernel and defect operator of the contractions and significantly enlarging the set of admissible functions. We also derive higher order trace formulas for maximal dissipative operators under relaxed assumptions and new simplified trace formulas for unitary and resolvent comparable self-adjoint operators. The respective spectral shift measures are absolutely continuous and, in the case of contractions, the set of admissible functions for the $n$th order trace formula on the unit circle includes the Besov class $B^n_{\infty, 1}(\T)$. Both aforementioned properties are new in the mentioned generality.
	\end{abstract}
	\maketitle
	
	\section{Introduction}
	
	Trace formulas for perturbed operator functions in terms of a spectral shift have a long history in operator theory and related fields. The concept originated from physics research summarized in \cite{Lif} and then developed to mathematical theory in the seminal works \cite{Kr53, Kr62}.  Initial trace formulas were derived for trace class perturbations $V$ of self-adjoint or unitary operators $H_0$ and allowed to efficiently compute the perturbed operator function $f(H_0+V)$ in terms of the initial data, namely,
	\begin{align}
		\label{babytf}
		\Tr(f(H_0+V)-f(H_0))=\int f'(\lambda)\xi_{H_0,V}\,d\lambda.
	\end{align}
	The first order spectral shift function $\xi_{H_0,V}$ controls the noncommutativity of operators $H_0$ and $V$ and is independent of the scalar function $f$. It is also closely related to important objects of perturbation theory, mathematical physics and noncommutative geometry including scattering phase, spectral flow, spectral action (see, e.g., \cite{BP,Skbook}). Those remarkable connections inspired  search for generalizations of the spectral shift and extensions of the trace formulas to models involving non-trace class perturbations and/or  nonnormal operators.
	
	An extension of the result of \cite{Kr53} to Schatten-von Neumann perturbations turned out to be a highly nontrivial task. In the self-adjoint case, the natural replacement of the left-hand side of \eqref{babytf} with an operator analog of a Taylor polynomial was successfully handled in \cite{Ko} for Hilbert-Schmidt perturbations, but the general case required development of a more subtle noncommutative analysis due to the intricate structure of Taylor remainders and was finally resolved in \cite{PoSkSu13In}.
	Subsequent generalizations of trace formulas to Dirac and Schrödinger operators (see, e.g., \cite{CNP,NS_2022,PoSkSu15}) had to overcome an extra challenge of noncompact perturbations while inclusions of unbounded dissipative %which model open systems,
	and contractive operators had to circumvent limitations of the spectral theory of nonnormal operators. As an outcome of a many-decade investigation and further advancement of noncommutative analysis, we have trace formulas for non-trace class perturbations of self-adjoint and unitary operators as well as trace class perturbations of contractive and maximal dissipative operators (see, e.g., \cite{Skbook} and references cited therein). The objective of this paper is to establish trace formulas in the missing cases of non-trace class Schatten-von Neumann perturbations of contractive and  resolvent comparable maximal dissipative operators.
	
	We note that extensions to noncompact perturbations progressed concurrently with trace formulas for unitary operators due to an intrinsic connection between the two cases which was noted in \cite{Kr62} and further developed in \cite{NH,PoSkSu15,PoSkSujfa2016,Skadv2017}.
	Trace formulas for pairs of unitaries with the trace class difference were derived in \cite{Kr62} and for unitaries with the Hilbert-Schmidt difference in \cite{NH,GePu}. The approach of \cite{NH} involved differentiation along a multiplicative path of unitaries while the approach of \cite{GePu} was based on differentiation along the linear path of contractions joining a pair of unitaries, which reduced the class of admissible functions. Higher order trace formulas were obtained in \cite{PoSkSujfa2016} for a pair of unitaries $U_0$ and $U_1=e^{iA}U_0$, with $A=A^*$ an element of the $n$th Schatten-von Neumann ideal, correcting analogous formulas previously obtained in \cite{Pe06}. The sets of functions satisfying the respective trace formulas were substantially enlarged in \cite{Skadv2017} to encompass all functions with the $n$th derivatives in the Wiener class except low degree polynomials. The aforementioned result of \cite{Skadv2017} is a starting point for the main results of this paper. %summarized below.
	
	Let $\hil$ denote a separable infinite-dimensional Hilbert space, $\mathcal{B}(\hil)$ the algebra of bounded linear operators on $\hil$, $\mathcal{B}_n(\hil)$ the $n$th Schatten-von Neumann ideal of compact operators on $\hil$ (see, e.g., \cite{GK,Sim} for a detailed discussion of their properties), and $\text{Tr}$ the canonical trace on the trace class ideal $\mathcal{B}_1(\hil)$. As usually, $\mathbb{N}$, $\mathbb{Z}$, $\mathbb{R}$, and $\mathbb{C}$ represent the sets of natural, integer, real, and complex numbers, respectively; $\mathbb{D}$ stands for the open unit disk and $\mathbb{T}$ for the unit circle in $\mathbb{C}$.  Positive constants are denoted by $c, d, \tilde{d}$ with subscripts indicating their dependencies; for example, $c_k$ depends only on $k$. Let $C(\mathbb{T})$ denote the Banach space of all continuous functions on $\mathbb{T}$ equipped with the supremum norm, $C^n(\mathbb{T})$ the space of all $n$-times continuously differentiable functions on $\mathbb{T}$, and $D^n(\mathbb{T})$ the space of all $n$-times differentiable functions on $\mathbb{T}$. Let $\dm$ denote the normalized arc length measure on $\cir$. Let $\hat{f}(k)$ denote the $k$th Fourier coefficient of $f\in C(\mathbb{T})$, that is,
	\[\hat{f}(k)=\frac{1}{2\pi i}\int_\cir f(z)\bar{z}^{k+1}dz,\quad k\in\mathbb{Z},\]
	and set
	\begin{equation}\label{FnT}
	\Fcal_n(\mathbb{T}):=\Big\{f(z)=\sum_{k=-\infty}^{\infty}\hat{f}(k)z^k\in C^n(\mathbb{T}):\;\sum_{k=-\infty}^{\infty}|k|^n|\hat{f}(k)|<\infty  \Big\}.
	\end{equation}
	The following result of \cite[Theorem 4.4]{Skadv2017} provides a generic $n$th order trace formula for pairs of unitary operators in terms of the spectral shift function $\eta_n$.
	When a perturbation is not compact, operators are not normal, or a class of admissible functions is extended, modifications of the approximating expressions and respective remainders can become necessary (see, e.g., \cite{NH,Skadv2017} and results of this paper).

	\begin{thm}%(\cite[Theorem 4.4]{Skadv2017})
		\label{skripkaunitaryadv}
		Let $n\in\mathbb{N},~n\geq 2$. Let $U_0$ be a unitary operator, $A=A^*\in\bnh$ and denote $U_s=e^{isA}U_0,~s\in [0,1]$. Then,
		%$\mathcal{R}_{U_0,f,n}(V)\in \mathcal{B}_1(\mathcal{H})$ for every $f\in\tFcal_n(\mathbb{T})$ and
		there exists a constant $c_n$ and a function  $\eta_n=\eta_{n,U_0,A}\in L^1(\cir,\dm)$  satisfying $$\|\eta_n\|_1\leq c_n\|A\|_n^n$$ such that
		\begin{align}
			\label{trfor}
			\textup{Tr}\left(f(U_1)-f(U_0)-\sum_{k=1}^{n-1}\dfrac{1}{k!}\dfrac{d^k}{ds^k}\Big|_{s=0}f(U_s) \right)=\int_{\cir}f^{(n)}(z)\eta_n(z)dz
		\end{align}
		for every $f\in\Fcal_n(\mathbb{T})$ with $\hat{f}(k)=0$ for $k=1,\dots,n-1$.
	\end{thm}

We note that the condition ``$\hat{f}(k)=0$ for $k = 1,\ldots, n-1$" was mistakenly not included in the statement of \cite[Theorem 4.4]{Skadv2017}.
The result of \cite[Theorem 4.4]{Skadv2017} is extended to all functions $f\in\Fcal_n(\mathbb{T})$ in Theorem~\ref{extuniatry} of this paper by modifying the respective trace formulas and by handling the low-degree polynomials differently.
The necessity of this new approach to the low-degree polynomials is explained in the proof of Theorem~\ref{extuniatry}.
We also note that the restriction ``$\hat{f}(k)=0$ for $k = 1,\ldots, n-1$" should be incorporated into the statement of \cite[Theorem~5.3]{Skadv2017} since that theorem relies on \cite[Theorem 4.4]{Skadv2017}. This restriction is no longer required in Theorem \ref{Trace formula: self-adjoint multiplicative path}, where it is eliminated using Theorem~\ref{extuniatry}.
The condition ``$\hat{f}(k)=0$ for $k = 1,\ldots, n-1$" should also be assumed in the statement of \cite[Theorem~3.2]{Ch_Pr_NYJM}, as it depends on \cite[Theorem 4.4]{Skadv2017}, and incorporated into the statements of \cite[Theorems~4.1 and~5.2]{Ch_Pr_NYJM}.

%Theorem~\ref{extuniatry} removes the restriction ``$\hat{f}(k)=0$ for $k = 1,\ldots, n-1$" of Theorem \ref{skripkaunitaryadv} by taking a different approach to the previously missing polynomial case, whose necessity is explained in its proof. The condition ``$\hat{f}(k)=0$ for $k = 1,\ldots, n-1$" should be incorporated in the statement of \cite[Theorem~5.3]{Skadv2017}, which relies on \cite[Theorem 4.4]{Skadv2017}. The latter oversight is rectified in the statement of Theorem \ref{Trace formula: self-adjoint multiplicative path} and discussed below. The condition ``$\hat{f}(k)=0$ for $k = 1,\ldots, n-1$" should also be assumed in the statement of \cite[Theorem~3.2]{Ch_Pr_NYJM} and incorporated in the statements of \cite[Theorems~4.1 and~5.2]{Ch_Pr_NYJM}, as noted below.

%{\cred Our manuscript explicitly addresses this previously missing polynomial case in the proof of Theorem~\ref{extuniatry}, providing a detailed explanation of why this case requires separate treatment. We also note that \cite[Theorem~4.4]{Skadv2017} is used in the proof of \cite[Theorem~5.3]{Skadv2017}. Since the aforementioned condition is necessary for \cite[Theorem~4.4]{Skadv2017} to hold, it should also be assumed in \cite[Theorem~5.3]{Skadv2017}. Furthermore, \cite[Theorem~4.4]{Skadv2017} is also used in \cite{Ch_Pr_NYJM}---explicitly in the proof of \cite[Theorem~3.2]{Ch_Pr_NYJM}, and thus implicitly in the proofs of \cite[Theorems~4.1 and~5.2]{Ch_Pr_NYJM}. Therefore, the same condition should be understood as required in the statements of those theorems in \cite{Ch_Pr_NYJM} as well.}
	
	The study of the first order trace formula and the associated spectral shift function on $\cir$ for a pair of contractions $T_0,T_1$ with the trace class difference was initiated in  \cite{Lan,Neid_1988} and further developed in a series of papers (see, e.g., the list of references in \cite[Section 5.5.2]{Skbook}).
	In all those attempts, additional assumptions on the associated defect operators and/or on $T_0$ were imposed. Those restrictions were ultimately removed in \cite{MNP19}.
	Higher order trace formulas \eqref{trfor} for contractions with difference in $\mathcal{B}_n(\hil)$ were established in \cite{PoSu} for $n=2$ and in \cite{PoSkSu} for $n\ge 3$, where the set of admissible functions $f$ was constrained to polynomials. The set of admissible functions was enlarged in \cite{Mor, Ch_Pr_NYJM} at a price of imposing stringent assumptions on the kernel and defect operator of the contractions (as recalled in Section \ref{trmultpath}).
	
	In Theorem \ref{extcontracmulti} of this paper we significantly relax the assumptions on operators made in \cite{Mor, Ch_Pr_NYJM} and enlarge the set of admissible functions obtained in \cite[Theorem 4.4]{Skadv2017} (see Theorem \ref{skripkaunitaryadv}) and \cite[Theorem 4.1]{Ch_Pr_NYJM}. More precisely, we establish a modification of \eqref{trfor} for contractions $U_0,U_1$ satisfying $U_1=e^{iA}U_0$ with $A=A^*\in\bnh$ for every $f$ in the set $\Fpt{n}$ (see Section \ref{prel}) containing the Besov class $B^n_{\infty, 1}(\T)$.
	Moreover, all the spectral shift measures in Theorem~\ref{extcontracmulti} are absolutely continuous and the densities of those not appearing in \eqref{trfor} due to the assumption $\hat{f}(k)=0$ for $k=1,\dots,n-1$ are trigonometric polynomials of degree at most $n-1$. The latter result (in fact, its special case established in Theorem \ref{extuniatry}) also greatly refines the trace formula of \cite[Remark 4.7(ii)]{PoSkSujfa2016} for unitaries, where the absolute continuity was not confirmed for some of the measures and the admissible class of functions was restricted to those $f$ in $\Fcal_n(\mathbb{T})$ for which $\hat{f}(k)=0$, $k<0$.
	
	The generalization of the first order trace formulas to pairs of maximal dissipative operators was initiated in \cite{AdPav} and subsequently investigated by several authors (see, e.g., the list of references in \cite[page 164]{Skbook}).
	In full generality the first order trace formula for pairs of maximal dissipative operators with the trace class resolvent difference was established in \cite{MNP19}. Second-order trace formulas for pairs of maximal dissipative operators with some stringent assumptions (as recalled in Section \ref{resolventcomtr}) were obtained in \cite{Mor} and analogous results in the higher order setting  were obtained in \cite{Ch_Pr_NYJM} for $\psi(\lambda)=f\left(\frac{\lambda+i}{\lambda-i}\right)$ such that $ f\in\Fcal_n(\cir)$ with $\hat{f}(k)=0$ for $k=1,\dots,n-1$.
	
	In Theorem \ref{Trace formula: self-adjoint multiplicative path} we significantly relax the assumptions on the maximal dissipative operators made in \cite{Mor, Ch_Pr_NYJM}
	and enlarge the set of admissible functions to
	\begin{align*}
		\FptR{n}:=&\left\{\psi(\lambda)=f\left(\frac{\lambda+i}{\lambda-i}\right):\; f\in \Fpt{n}\right\},
	\end{align*}
	which contains all rational functions bounded on $\R$ (see, e.g., \cite[Section 5]{Skadv2017}). The result of Theorem~\ref{Trace formula: self-adjoint multiplicative path} also extends the trace formula of \cite[Theorem 5.2]{Ch_Pr_NYJM} for maximal dissipative operators and the trace formula of \cite[Theorem 5.3]{Skadv2017} for self-adjoint operators $H_0$ and $H_1$ satisfying $(H_1-iI)^{-1}-(H_0-iI)^{-1}\in\bnh$ to include $$\text{span}\Big\{\R\ni\lambda\mapsto\frac{1}{a-\lambda}: \operatorname{Im}(a)>0\Big\}$$
and $\psi\in\FptR{n}$ arising from $f\in\Fpt{n}\setminus\Fcal_n(\cir)$ into the set of admissible functions.
	
	A variant of the first order trace formula for contractions with the integration going over $\mathbb{D}$ was obtained in \cite{AC_KBS_JOT}.
	The right-hand side of the latter formula resembles the one of the Helton-Howe formula (see, e.g., \cite{HH,CS}). In Theorem \ref{Helt1} we obtain a higher order analog of the Helton-Howe type trace formula \cite{AC_KBS_JOT}.
	
	In Theorem~\ref{linunitarytraceformula} we significantly simplify the higher order trace formula \eqref{trfor} for unitary operators by replacing the left-hand side with an alternative approximation remainder and in Theorem~\ref{Trace formula: self-adjoint linear path} we simplify the trace formula of \cite[Theorem 5.3]{Skadv2017} for resolvent-comparable self-adjoint operators. The new $n$th order trace formula for unitaries holds for $f\in\Fpt{n}$ and the new $n$th order trace formula for self-adjoints holds for $f\in\FptR{n}$, considerably enlarging the admissible function classes attained in the analogous results of \cite[Theorem 2.6 and Theorem 3.5]{PoSkSu15}, respectively.
	
	Our major tools include multilinear operator integration and Sch\"affer's unitary matrix dilation of contractions. In particular, we utilize norm bounds, perturbation formulas and change of variables techniques for multilinear operator integrals as well as dilate operator Taylor remainders from the case of contractions to the case of unitaries. The aforementioned methods are synthesized along a carefully selected path between the initial and perturbed operators to remove prior restrictive assumptions.
	
	The paper is organized as follows: preliminaries on multilinear operator integration  and Sch\"affer's unitary dilation are collected in Section \ref{prel}, higher order trace formulas for contractions are established in Section \ref{trmultpath}, higher order trace formulas for maximal dissipative operators are derived in Section \ref{resolventcomtr} and simplified higher order trace formulas for unitaries and resolvent comparable self-adjoint operators are derived in Section \ref{simpletrformula}.

	\section{Preliminaries}\label{prel}
	
	In this section we recall necessary facts on multilinear operator integration   and Sch\"affer's unitary dilation technique.

The symbols of multiple operator integrals utilized in this paper are constructed from divided differences. We recall that
the zeroth-order divided difference of a function $f$ is simply the function itself, denoted by $f^{[0]}:=f$. Consider points $z_0, z_1, \ldots, z_n$ in $\cir$ and let $f\in  D^n(\cir)$. The divided difference $f^{[n]}$ of order $n$ is defined recursively as follows:
	\begin{align*}
		f^{[n]}(z_0,z_1,\ldots,z_n):=\begin{cases}
			\frac{f^{[n-1]}(z_0,z_2,\ldots,z_n)-f^{[n-1]}(z_1,z_2,\ldots,z_n)}{z_0-z_1} & \text{if}\ z_0\neq z_1,\\
			\frac{\partial}{\partial z_1}f^{[n-1]}(z_1,\ldots,z_n) &\text{if}\ z_0=z_1.
		\end{cases}
	\end{align*}
	
\paragraph{\bf New function class.} Let $n\in\Nat$. Let $\Fpt{n}$ be the collection of all functions $f\in D^n(\cir)$ such that $f^{(n)}$ is bounded and $f^{[n]}$ can be expressed as
\begin{align}\label{divdef1}
f^{[n]}(z_0,\ldots,z_n)=\int_{\Omega}\,a_0(z_0,\omega)\cdots a_n(z_n,\omega)\,d\nu(\omega),
\end{align}
where $(\Omega, d\nu)$ is a $\sigma$-finite measure space and  \[a_i(\cdot,\cdot):\;\cir\times\Omega\to\mathbb{C},\quad i=0,\dots,n,\]
are bounded measurable functions satisfying
\begin{align}\label{divdef1b}
\int_{\Omega}\|a_0(z_0,\omega)\|_{\infty} \cdots \|a_n(z_n,\omega)\|_{\infty}\, d|\nu|(\omega)<\infty.
\end{align}
Consider
\begin{align}\label{divdef2}
\|f^{[n]}\|_{\pt{n}}:=\inf \int_{\Omega}\|a_0(z_0,\omega)\|_{\infty} \cdots \|a_n(z_n,\omega)\|_{\infty}\, d|\nu|(\omega)<\infty,
\end{align}
where the infimum is taken over all possible representations \eqref{divdef1}.
More generally, $\|\cdot\|_{\pt{n}}$ is a norm on the algebra of functions on $\T^{n+1}$ admitting the representation \eqref{divdef1} (see, e.g., \cite{PagSuko04}).

Below we relate the newly introduced set $\Fpt{n}$ to function classes previously considered in the context of trace formulas on $\T$. One of those classes is $\Fcal_n(\mathbb{T})$ defined in \eqref{FnT}
and the other is the Besov class $B^n_{\infty, 1}(\T)$ defined as follows. Let $w$ be an infinitely differentiable function on $\R$ such that
$$w\ge 0, \ \  \supp w \subset\left[\frac12,2\right] , \ \ \text{and} \ \
 w(x) = 1 - w \left(\frac{x}2\right) \ \ \text{for} \ \
 x\in[1,2].$$
Consider the trigonometric polynomials $W_m,$ and $W_m^\sharp$ defined on $\T$ by
$$W_m(z)=\sum_{k\in\mathbb{Z}}w \left(\frac{k}{2^m}\right)z^k,  \ \ m\ge 1, \ \
W_0(z)=\overline{z}+1+z, \ \ \text{and} \ \
W_m^\sharp(z)=\overline{W_m(z)}, \ \ m\ge 0.$$ Then, for each
function $\varphi$ on $\mathbb{T}$,
$$\varphi=\sum_{m\ge 0}\varphi\ast W_m+\sum_{m\ge 1}\varphi\ast W_m^\sharp.$$
The Besov class $B_{\infty, 1}^n(\T)$ consists of functions  $\varphi$
on $\mathbb{T}$ such that
$$\{\|2^{nm}\varphi\ast W_m\|_{\infty}\}_{m\ge 0}\in \ell_1 \ \
\text{and} \ \ \{\|2^{nm}\varphi\ast W_m^\sharp\|_{\infty}\}_{m\ge
1}\in \ell_1.$$

\begin{prop}\label{new_class}
Let $n\in\N$. Then, the following inclusions hold.
\begin{enumerate}[(i)]
\item\label{new_class_i} $B^n_{\infty,1}(\mathbb{T})\subset\Fpt{n}$.

\item\label{new_class_ii} $\Fpt{n}\subset \Fcal_{n-1}(\mathbb{T})$.

\item\label{new_class_iii} $\Fpt{n} = \bigcap_{k=1}^{n} \Fpt{k}$.
\end{enumerate}
\end{prop}

\begin{proof}
(i) Follows from \cite[Theorem 4.4]{Pe06}. %\cite[Theorem 2.2]{PoSkSujfa2016})

(ii) Let $f\in\Fpt{n}$. Note that the Fourier series of $f^{(k)}$ is absolutely convergent for $0 \leq k \leq n-1$. Indeed, we have
\begin{align*}
\widehat{f^{(k)}}(l)=\frac{1}{2\pi i} \int_{\mathbb{T}} f^{(k)}(z)\bar{z}^{l+1} dz = \frac{1}{2\pi} \int_{0}^{2\pi} f^{(k)}(e^{i\theta}) e^{-il\theta} d\theta.
\end{align*}
Integrating by parts yields
\begin{align*}
\widehat{f^{(k)}}(l)=\frac{1}{2\pi i l} \int_{0}^{2\pi} f^{(k+1)}(e^{i\theta}) (i e^{i\theta}) e^{-il\theta} d\theta=\frac{1}{2\pi il} \int_{\mathbb{T}} f^{(k+1)}(z)\bar{z}^{l} dz,\quad l\neq 0.
\end{align*}
Thus, we obtain $l\widehat{f^{(k)}}(l)=\widehat{f^{(k+1)}}(l-1)$. Since $f^{(k+1)} \in L^{\infty}(\mathbb{T}) \subset L^2(\mathbb{T})$, we have
\begin{align*}
\sum_{l \in \mathbb{Z}\setminus\{0\}}\left| \widehat{f^{(k)}}(l) \right|
=\sum_{l \in \mathbb{Z}\setminus\{0\}} \frac{1}{|l|} \left| \widehat{f^{(k+1)}}(l-1) \right|
\leq \Big( \sum_{l \in \mathbb{Z}\setminus\{0\}} \frac{1}{l^2}\Big)^{1/2}
\Big( \sum_{l \in \mathbb{Z}\setminus\{0\}} \left| \widehat{f^{(k+1)}}(l-1) \right|^2 \Big)^{1/2} < \infty.
\end{align*}
Consequently, for $n\ge 2$,
\begin{align*}
\sum_{l\in\mathbb{Z}\setminus\{0\}}|l|^{n-1}|\hat{f}(l)|
&\le n^{n-1}\!\sum_{l\in\mathbb{Z}\setminus\{0\}}|l\cdots(l-n+2)|\,|\hat{f}(l)|\\
&=n^{n-1}\!\sum_{l \in \mathbb{Z}\setminus\{0\}} |(l+1)\cdots (l+n-1)|\, |\hat{f}(l+n-1)|\\
&=n^{n-1}\!\sum_{l\in\mathbb{Z}\setminus\{0\}}\left| \widehat{f^{(n-1)}}(l)\right|<\infty,
\end{align*}
implying $f\in\Fcal_{n-1}(\mathbb{T})$.

(iii) It follows from \eqref{new_class_i} and \eqref{new_class_ii} that $f \in \Fpt{k}$, implying the result.
\end{proof}

\paragraph{\bf {Multilinear operator integration.}}		
The subsequent definition provides a simple yet widely applicable formulation of the multilinear operator integral, as outlined in \cite{PoSkSujfa2016, Skadv2017}.
	\begin{dfn}\label{moi}
		Let $f\in\Fpt{n}$. Let $k\in\{1,\ldots,n\}$. Let $1 \leq \alpha, \alpha_i \leq \infty$ for $i = 1,\ldots, k$ be such that $\frac{1}{\alpha_1} + \cdots + \frac{1}{\alpha_k} = \frac{1}{\alpha}$. Let $U_i$, $i= 0, ..., k,$ be unitary operators on $\hil$. The mapping
\[T_{f^{[k]}}^{ U_0,\ldots,U_k}: \mathcal{B}_{\alpha_1}(\hil)\times\cdots\times\mathcal{B}_{\alpha_{k}}(\hil)\longrightarrow \mathcal{B}_{\alpha}(\hil)\]	
defined by
\begin{align*}
T_{f^{[k]}}^{ U_0,\ldots,U_k}(V_1,\ldots,V_k)=\int_\Omega a_0(U_0,\omega)V_1a_1(U_1,\omega)\cdots V_ka_k(U_k,m)\,d\nu(\omega)
\end{align*}
where $(\Omega,\nu)$ and $a_i(\cdot,\cdot)$ satisfy \eqref{divdef1} and \eqref{divdef1b}, is called a multilinear operator integral with symbol $f^{[k]}$.  The mapping $T_{f^{[k]}}^{ U_0,\ldots,U_k}$ is independent of the choice of $(\Omega,\nu)$, $a_i(\cdot,\cdot)$ in the decomposition \eqref{divdef1} and $\|T_{f^{[k]}}^{ U_0,\ldots,U_k}\|\le\|f^{[k]}\|_{\pt{k}}$.
	\end{dfn}
	
The following properties of multilinear operator integrals are essential in proving our results.
\begin{thm}\label{moiest}
Let $n\in\Nat$, let $k\in\{1,\ldots,n\}$. Let $1 < \alpha,  \widetilde\alpha, \alpha_i < \infty$ for $i = 1,\ldots, k$ be such that $\frac{1}{\alpha_1} + \cdots + \frac{1}{\alpha_k} = \frac{1}{\alpha}$ and $\frac{1}{\alpha_1} + \cdots + \frac{1}{\alpha_{k-1}} = \frac{1}{\widetilde\alpha}$. Let $U_i$, $i= 0, ..., k,$ be unitary operators on $\hil$  and let $f\in\Fpt{n}$. Then, the following assertions hold.
{\vspace{.1in}}
		
\noindent$(i)$ The transformation $T_{f^{[k]}}^{ U_0,\ldots,U_k}: \mathcal{B}_{\alpha_1}(\hil)\times\cdots\times\mathcal{B}_{\alpha_{k}}(\hil)\longrightarrow \mathcal{B}_{\alpha}(\hil)$  is bounded and
\[\|T_{f^{[k]}}^{ U_0,\ldots,U_k}\|\leq c_{k,\alpha_1,\ldots, \alpha_k}\|f^{(k)}\|_\infty.\]
		
\noindent $(ii)$  Let $\phi_k(z_0,z_1,\ldots,z_{k-1}):=f^{[k]}(z_0,z_1,\ldots,z_{k-1},z_0)$. The transformation $T_{\phi_k}^{ U_0,\ldots,U_{k-1}}: \mathcal{B}_{\alpha_1}(\hil)\times\cdots\times\mathcal{B}_{\alpha_{k-1}}(\hil)\longrightarrow \mathcal{B}_{\widetilde\alpha}(\hil)$  is bounded and
\[\|T_{\phi_k}^{ U_0,\ldots,U_{k-1}}\|\leq c_{k,\alpha_1,\ldots, \alpha_{k-1}}\|f^{(k)}\|_\infty.\]
\end{thm}
\begin{proof}
The result follows from \cite[Proposition 7.4]{CoLemSu21} and \cite[Theorem 3.3]{Co24} along with the fact that, for $2 \leq \alpha_i < \infty$, $\mathcal{B}_2(\hil) \cap \mathcal{B}_{\alpha_i}(\hil)$ is dense in $\mathcal{B}_{\alpha_i}(\hil)$.
% and Lemma \ref{equality_moi}.
\end{proof}
\begin{crlre}\label{useofcyclicity}
		Let $n\in\Nat$ and $f\in\Fpt{n}$. Let $U_0, U_1$ be unitary operators on $\hil$ and let $V_1,V_2,\ldots,V_n\in\bnh$. Then, the following assertions hold.
		\begin{enumerate}[(i)]
			\item%\label{derreduce}
			For $n=1$,
			\[\Tr\left(T_{f^{[1]}}^{\,U_0,U_0}(V_1)\right)=\Tr\left(f'(U_0)V_1\right)\]
			and
			\begin{align}
				\label{useofcyclicity1}
				\left|\Tr\left(T_{f^{[1]}}^{\,U_0,U_0}(V_1)\right)\right|\leq c_1\|f'\|_\infty\,\|V_1\|_1.
			\end{align}
			\item  For $n\geq 2$,
			\[\Tr\left(T_{f^{[n]}}^{U_0,U_1,U_0,\ldots,U_0}(V_1,\ldots,V_n)\right)
			=\Tr\left(T_{\widetilde{f^{[n]}}}^{U_0,U_1,U_0,\ldots,U_0}(V_1,\ldots,V_{n-1})\, V_n\right)\]
			where $\widetilde{f^{[n]}}(z_0,\ldots,z_{n-1})=f^{[n]}(z_0,\ldots,z_{n-1},z_0)$. Furthermore,
			\begin{align}
				\label{useofcyclicity2}
				\left|\Tr\left(T_{f^{[n]}}^{U_0,U_1,U_0,\ldots,U_0}(V_1,\ldots,V_n)\right)\right|\leq c_n\|f^{(n)}\|_\infty\prod_{k=1}^{n}\|V_k\|_n.
			\end{align}
		\end{enumerate}
	\end{crlre}
	
\begin{proof}
If $n\geq 2$, it follows from Definition \ref{moi}, a minor adjustment of the argument in the proof of \cite[Corollary 4.8]{AzCaDoSu09}, and cyclicity of the trace that
\begin{align*}
&\Tr\left(T_{f^{[n]}}^{U_0,U_1,U_0,\ldots,U_0}(V_1,\ldots,V_n)\right)\\
=&\Tr\left(\int_\Omega a_0(U_0,\omega)\,V_1\,a_1(U_1,\omega)\cdots\, V_n\,a_n(U_0,\omega)\,d\nu(\omega)\right)\\
=&\Tr\left[\left(\int_\Omega a_n(U_0,\omega)a_0(U_0,\omega)\,V_1\,a_1(U_1,\omega)\cdots\, V_{n-1}\,a_{n-1}(U_0,\omega)\right)V_n\,d\nu(\omega)\right]\\
=&\Tr\left(T_{\widetilde{f^{[n]}}}^{U_0,U_1,U_0,\ldots,U_0}(V_1,\ldots,V_{n-1})\cdot V_n\right).
		\end{align*}
		Similarly, we obtain $\Tr\left(T_{f^{[1]}}^{\,U_0,U_0}(V_1)\right)=\Tr\left(f'(U_0)V_1\right)$. The estimates \eqref{useofcyclicity1} and \eqref{useofcyclicity2} follow from the application of Theorem \ref{moiest} and H\"{o}lder's inequality.
	\end{proof}
	
	The following perturbation formulas follow from the proof of \cite[Lemma 2.4(i)]{PoSkSujfa2016}.
	
	\begin{lma}\label{algebraicprop}
		Let $n\in\Nat$ and $f\in\Fpt{n}$. Let $U_0, U_1, U_2$ be unitary operators on $\hil$ and let $V_1,V_2,\ldots,V_{n-1}\in\mathcal{B}(\hil)$. Then,
		\begin{enumerate}[(i)]
			\item
			$f(U_1)-f(U_0)=T_{f^{[1]}}^{U_1,U_0}(U_1-U_0)=T_{f^{[1]}}^{U_0,U_1}(U_1-U_0)$,
			\item for $n\geq 2$,
			\begin{align*}			T_{f^{[n-1]}}^{U_0,U_1,U_0,\ldots,U_0}(V_1,\ldots,V_{n-1})-T_{f^{[n-1]}}^{U_0,U_2,U_0,\ldots,U_0}(V_1,\ldots,V_{n-1})\\
				=T_{f^{[n]}}^{U_0,U_1,U_2,U_0,\ldots,U_0}(V_1,U_1-U_2,V_2,\ldots,V_{n-1}).
			\end{align*}
		\end{enumerate}
	\end{lma}
	
The existence of the G\^ateaux derivative in Theorem \ref{derivativeformulaunitarypath} below and the representation \eqref{25} are proved in  \cite[Theorem 5.1.]{Co24}. The continuity of the operator derivative under the assumption $f^{(n)}\in C^n(\T)$ is due to \cite[Corollary 3.6]{ChCoGiPr24}.
The formula \eqref{uniderfor} follows from \eqref{25} and properties of the divided difference.
	
\begin{thm}\label{derivativeformulaunitarypath}
Let $1<p<\infty$. Let $A\in\mathcal{B}_p(\hil)$ be a self-adjoint operator, $U_0$ a unitary operator, and $ U_t=e^{itA}U_0$. Let $n\in\Nat$ and $f\in\Fpt{n}$. Then, the G\^ateaux derivative $\frac{d^n}{dt^n}\big|_{t=s}f( U_t)$ exists in the Schatten $p$-norm and admits the representation
\begin{align}\label{25}
\dfrac{d^n}{dt^n}\Big|_{t=s}f(U_t)			=i^n\sum_{r=1}^{n}~\sum_{\substack{l_1+l_2+\cdots+l_{r}=n\\l_1,l_2,\ldots,l_{r}\geq 1}}~\dfrac{n!}{l_1!\cdots l_r!}\,T_{f^{[r]}}^{ U_s,\dots, U_s}\left(A^{l_1}U_s,\ldots, A^{l_r}U_s\right).
\end{align}
Moreover if $f^{(n)}\in C(\cir)$, then $s\mapsto \frac{d^n}{dt^n}\big|_{t=s}f( U_t)$ is continuous in the Schatten $p$-norm.
In particular, for every $k\in\Nat$,
\begin{align}\label{uniderfor}
\dfrac{d^n}{dt^n}\Big|_{t=s}U_t^k			=\sum_{r=1}^{n}~\sum_{\substack{l_1+l_2+\cdots+l_{r}=n\\l_1,l_2,\ldots,l_{r}\geq 1}}~\dfrac{n!}{l_1!\cdots l_r!}\Bigg[\sum_{\substack{\alpha_0+\alpha_1+\cdots+\alpha_r=k\\\alpha_0\geq 0; \,\alpha_1,\ldots,\alpha_r\geq 1}}U_s^{\alpha_0}(iA)^{l_1}U_s^{\alpha_1}\cdots (iA)^{l_{r}}U_s^{\alpha_r}\Bigg].
\end{align}
\end{thm}
We will also need the following estimate.
\begin{lma}\label{explip0}
Let $n\in\mathbb{N}$. Let $A, B$ be two self-adjoint operators such that $A-B\in\mathcal{B}_n(\mathcal{H})$. Then $\|e^{iA}-e^{iB}\|_n \leq e^{\max\{\|A\|,\|B\|\}} \|A-B\|_n$.
\end{lma}
\begin{proof}
	By the power series expansions for $e^{iA}$ and $e^{iB}$ and by telescoping, we have
	\begin{align*}
		\|e^{i A}-e^{iB}\|_n\leq \sum_{k=1}^{\infty}\sum_{ p=0}^{k-1}\frac{\|(iA)^p(iA-iB)(iB)^{k-1-p}\|_n}{k!}
		\leq e^{\max\{\|A\|,\|B\|\}}\,\|A-B\|_n.
	\end{align*}
\end{proof}

\begin{lma}\label{lem:moi_rep}
Let $n \in \Nat$, $n\geq 2$, and $ f\in\Fpt{n}$. Let $A\in\mathcal{B}_n(\hil)$ be a self-adjoint operator, $U_0$ a unitary operator, and $U_t = e^{itA}U_0$, $t\in [0,1]$. Then,
			\begin{align}\label{moi_rep}
				f(U_1) - f(U_0) - \sum_{k=1}^{n-1} \frac{1}{k!} \frac{d^k}{ds^k} \Big|_{s=0} f(U_s) = \frac{1}{(n-1)!} \int_{0}^{1} (1-t)^{n-1} \, \frac{d^n}{ds^n} \Big|_{s=t} f(U_s) \, dt,
			\end{align}
			where the integral on the right-hand side converges in the $\|\cdot\|_n$-norm. Moreover, the following properties hold:
\begin{align}
&\label{mes1}\int_{0}^{1} (1-t)^{n-1} \, \frac{d^n}{ds^n} \Big|_{s=t} f(U_s) \, dt\in\boh,\\
&\label{mes2}
t\mapsto(1-t)^{n-1} \, \Tr\left(\frac{d^n}{ds^n} \Big|_{s=t} f(U_s)\right)
\text{is bounded and measurable on } [0,1],\\
&\label{tr_moi_rep}
\Tr\Bigg(f(U_1) - f(U_0) - \sum_{k=1}^{n-1} \frac{1}{k!} \frac{d^k}{ds^k} \Big|_{s=0} f(U_s)\Bigg) = \frac{1}{(n-1)!} \int_{0}^{1} (1-t)^{n-1} \, \Tr\left(\frac{d^n}{ds^n} \Big|_{s=t} f(U_s)\right) \, dt.
\end{align}
\end{lma}

\begin{proof}
Let $f\in\Fpt{n}$. For $t\in [0,1]$, define $\Gamma(t)=f(U_t)-f(U_0)$. By Lemma \ref{algebraicprop}, $\Gamma(t)=T_{f^{[1]}}^{U_t, U_0}(U_t-U_0)$, which along with observations made in Definition \ref{moi}, Lemma \ref{explip0}, implies
\begin{align*}
	\|\Gamma(t)\|_n\leq \|f^{[1]}\|_{\pt{1}} \|U_t-U_0\|_n=&\|f^{[1]}\|_{\pt{1}} \|e^{itA}-I\|_n\leq \|f^{[1]}\|_{\pt{1}}\,e^{\|A\|}\|A\|_n<\infty.
\end{align*}
Therefore, by Theorem \ref{derivativeformulaunitarypath}, $[0,1]\ni t\mapsto \Gamma(t)\in\bnh$ is $n$-times differentiable in the norm $\|\cdot\|_n$ and $\frac{d^k}{ds^k}\big|_{s=t}\Gamma(s)=\frac{d^k}{ds^k}\big|_{s=t}f(U_s)$ for $k=1,\dots, n$. Let $\psi\in (\bnh)^*$. Then, $[0,1]\ni t\mapsto \psi(\Gamma(t))$ is also $n$-times differentiable and
\begin{align*}
	\frac{d^k}{ds^k}\Big|_{s=t} \psi(\Gamma(s))= \psi\left( \frac{d^k}{ds^k}\Big|_{s=t}\Gamma(s)\right)=\psi\left( \frac{d^k}{ds^k}\Big|_{s=t}f(U_s)\right)\quad 1\leq k\leq n.
\end{align*}

It follows from \eqref{25} and Definition \ref{moi} that
\begin{align*}
	\left\| (1-t)^{n-1}\frac{d^n}{ds^n} \Big|_{s=t} f(U_s) \right\|_n \leq d_n \max_{1\leq k \leq n}\|f^{[k]}\|_{\pt{k}} \|A\|_n\,\|A\|^{n-1}
\end{align*}
for some positive constant $d_n$. Consequently,
\begin{align*}
	\left| (1-t)^{n-1}\frac{d^k}{ds^k}\Big|_{s=t} \psi(\Gamma(s)) \right| \leq d_n\, \|\psi\|\max_{1\leq k \leq n}\|f^{[k]}\|_{\pt{k}} \|A\|_n\,\|A\|^{n-1}.
\end{align*}
Applying the fundamental theorem of calculus to the function $[0,1]\ni t\mapsto \psi(\Gamma(t))$  and integrating by parts yields
\begin{align*}
	\psi\left(\Gamma(1)-\Gamma(0) - \sum_{k=1}^{n-1} \frac{1}{k!} \frac{d^k}{ds^k} \Big|_{s=0} \Gamma(s)\right)
	=\frac{1}{(n-1)!} \int_{0}^{1} (1-t)^{n-1}\,\frac{d^n}{ds^n}\Big|_{s=t}{\psi(\Gamma(s))}\,dt.
\end{align*}
Consequently,
\begin{align}
\label{psirem}
	\psi\left(f(U_1) - f(U_0) - \sum_{k=1}^{n-1} \frac{1}{k!} \frac{d^k}{ds^k} \Big|_{s=0} f(U_s)\right)
	=\frac{1}{(n-1)!} \int_{0}^{1} (1-t)^{n-1} \, \psi\left(\frac{d^n}{ds^n} \Big|_{s=t} f(U_s)\right) \, dt.
\end{align}
Observe that
$[0,1]\ni t \mapsto (1-t)^{n-1}\frac{d^n}{ds^n} \big|_{s=t} f(U_s)\in \mathcal{B}_n(\hil)$ is strongly Borel-measurable (see \cite[Definition V.4.1 (p. 130)]{Yosida}). Therefore, by \cite[Theorem V.5.1 (p. 133)]{Yosida},
\begin{align*}
	\int_{0}^{1} (1-t)^{n-1} \, \frac{d^n}{ds^n} \Big|_{s=t} f(U_s) \, dt
\end{align*}
exists in the $\|\cdot\|_n$-norm and, consequently,
\begin{align*}
	\psi\left(f(U_1) - f(U_0) - \sum_{k=1}^{n-1} \frac{1}{k!} \frac{d^k}{ds^k} \Big|_{s=0} f(U_s)\right) = \frac{1}{(n-1)!}\psi\left( \int_{0}^{1} (1-t)^{n-1} \, \frac{d^n}{ds^n} \Big|_{s=t} f(U_s)\, dt\right)
\end{align*}
for every $\psi\in (\bnh)^*$. The latter implies \eqref{moi_rep}.

Applying the representation for the derivative \eqref{25} and Definition \ref{moi} implies
\begin{align*}
	\Big\| (1-t)^{n-1}\frac{d^n}{ds^n} \Big|_{s=t} f(U_s) \Big\|_1 \leq d_n \max_{1\leq k \leq n}\|f^{[k]}\|_{\pt{k}} \|A\|_n^n.
\end{align*}
By a reasoning similar to the one above, $t\mapsto \left\langle(1-t)^{n-1}\frac{d^n}{ds^n}\big|_{s=t} f(U_s)h_1,h_2\right\rangle$ is Borel measurable for all $h_1,h_2\in\mathcal{H}$.
The latter two properties along with \cite[Proposition 3.2 and Lemma 3.10]{AzCaDoSu09} imply \eqref{mes1} and \eqref{mes2} and
\begin{align*}
\Tr\Big(\int_{0}^{1}(1-t)^{n-1}\,\frac{d^n}{ds^n}\Big|_{s=t}f(U_s)\,dt\Big)
=\int_{0}^{1}(1-t)^{n-1}\,\Tr\Big(\frac{d^n}{ds^n}\Big|_{s=t}f(U_s)\Big)\,dt.
\end{align*}
Combining the latter with \eqref{moi_rep} implies \eqref{tr_moi_rep}.
\end{proof}

\paragraph{\bf {Sch\"affer's unitary matrix dilation.}}
 Let $\hil$ be a Hilbert space, $\ell_2(\hil)=\oplus_{1}^{\infty}\hil$, and let $\mathcal{K}$ be a Hilbert space containing $\hil$ as a closed subspace.
We recall that a power unitary dilation of a contraction
$T\in\mathcal{B}(\hil)$ is a unitary $U\in\mathcal{B}(\mathcal{K})$ satisfying
$T^n=P\,U^n\upharpoonright_\hil$ for every $n\in\Nat$, where $P$ is the orthogonal projection from $\mathcal{K}$ onto $\hil$. We will use a power unitary dilation of a contraction explicitly constructed in \cite{schaffer} and summarized below.

Let $U_{T}$ be the unitary operator on $\ell_2(\hil)\oplus \hil\oplus \ell_2(\hil)$  with block matrix representation
\begin{align}\label{schafferdil}
	U_{T}=\begin{blockarray}{ccccc}
		\ell_2(\hil) & \hil & \ell_2(\hil) &  \\[4pt]
		\begin{block}{[ccc]cc}
			S^*& 0& 0 & \ell_2(\hil) \\[3pt]
			D_{T^*}P_\hil& T& 0 & \hil\\[3pt]
			-T^*P_\hil& D_{T}& S& \ell_2(\hil)\\[3pt]
		\end{block}
	\end{blockarray},
\end{align}
where $S$ is the unilateral shift on $\ell_2(\hil)$ given by $$S(h_1, h_2, \ldots,)=(0,h_1, h_2, \ldots),\qquad h_i\in\hil ,$$  $D_{T}=(1-T^*T)^{1/2}$,  $D_{T^*}=(1-TT^*)^{1/2}$  are the defect operators corresponding to the contractions $T$ and $T^*$, respectively,  and $P_\hil$ is the orthogonal projection from $\ell_2(\hil)$ onto $\hil\oplus 0\oplus0\oplus\cdots$. Then $U_T$ is a power unitary dilation of $T$, known as the Sch\"{a}ffer unitary matrix dilation of $T$.

Further details on dilation theory can be found in \cite{NFbook} and examples of its application in derivation of trace formulas can be found in \cite{AcSdCp24,DySk14,Mor,MNP19}.

%Interested readers may refer to \cite{NFbook} for more insight into dilation theory.
	
\section{Higher order trace formulas for contractions}
%\section{Trace formulae in multiplicative path}
\label{trmultpath}
	
In this section we establish higher order trace formulas for functions of contractions  $T_s$ without imposing the stringent assumptions of \cite[Theorem 4.1]{Ch_Pr_NYJM} that $\dim \ker(T_s)=\dim \ker(T_s^*)$ and that the defect operator  of $T_s$
belongs to $\bnh$.
The trace formulas in Theorem \ref{extcontracmulti} involve integration over $\cir$ and in Theorem \ref{Helt1} over $\mathbb{D}$.
	
We start by broadening the set of admissible functions satisfying \eqref{trfor} to include the set $\mathcal{P}_n(\cir)$ of polynomials of degree at most $n-1$ and by establishing the absolute continuity of the spectral shift measures in \cite[Remark 4.7(ii)]{PoSkSujfa2016}. Both goals are achieved in Theorem \ref{extuniatry} below.
	
	\begin{Not}\label{not_mult_uni}
		Let $n\in\Nat$, $n\geq 2$. Let $U_0$ be a unitary operator, $A$ a bounded self-adjoint operator on $\hil$, and $$U_s=e^{isA}U_0,\quad s\in [0,1].$$ Define
		\begin{align}
			\label{remua}
			\mathcal{R}_{n}^{\mult}(U_0,A,f)
			:=f(U_1)-f(U_0)-\sum_{k=1}^{n-1}\dfrac{1}{k!}\dfrac{d^k}{ds^k}\Big|_{s=0}f(U_s).
		\end{align}
	\end{Not}

\begin{thm}\label{extuniatry}
Assume Notations \ref{not_mult_uni} and assume that $A\in\mathcal{B}_n(\hil)$.
Then, $\mathcal{R}_{n}^{\mult}(U_0,A,f)\in \mathcal{B}_1(\hil)$ for every $f\in\Fpt{n}$. Furthermore, there exist a constant $d_{n}>0$ and functions $\eta_n\in  L^1(\cir,\dm)$ and $\eta_k\in{\rm span}\{ \overline{z},\dots,\bar{z}^{\,n-k}\}$, $k=1,\dots,n-1$, satisfying
		$$\|\eta_k\|_1\leq d_{n}\|A\|_n^n,\quad k=1,\dots,n,$$ such that \begin{align}\label{trformulaunigen}
			\Tr\left(\mathcal{R}_{n}^{\mult}(U_0,A,f) \right)=\sum_{k=1}^{n}\int_{\cir}f^{(k)}(z)\eta_k(z)dz
		\end{align}
		for every $f\in\Fpt{n}$. Moreover, if $\hat f(k)=0$ for $k=1,\dots,n-1$, then \eqref{trformulaunigen} holds with $\eta_k=0$ for $k=1,\dots,n-1$.
	\end{thm}

	\begin{proof}
Denote $$\tFpt{n}:=\{f\in\Fpt{n}: \hat{f}(k)=0 \text{ for } k=1,\dots,n-1\}.$$
The representation \eqref{trformulaunigen} with $\eta_1=\dots=\eta_{n-1}=0$ for $f\in \tFpt{n}$ can be established along the same lines as the proof of \cite[Theorem 4.4]{Skadv2017} by applying Theorem \ref{moiest},
Corollary \ref{useofcyclicity}, and the integral representation for the Taylor remainder \eqref{tr_moi_rep}.

We note that the proof of \cite[Theorem 4.4]{Skadv2017} follows the proof of \cite[Theorem 4.1]{PoSkSujfa2016} and uses the representation \cite[(4.12)]{PoSkSujfa2016} for functions $f\in\Fcal_n(\mathbb{T})$ satisfying $\int_0^{2\pi}f^{(k)}(e^{it})\,dt=0$ for all $k=1,\dots,n-1$. The functions not satisfying the latter condition should be treated differently. Below we present the proof of \eqref{trformulaunigen} for $f\in\Fpt{n}$.
		
		Firstly, we establish
		\begin{align}\label{a3}
			\Tr\left(\mathcal{R}_{n}^{\mult}(U_0,A,f) \right) =\sum_{k=1}^{n-1}\int_{\cir}\,f^{(k)}(z)\,\eta_k(z)\,dz
		\end{align} for every $f\in\mathcal{P}_n(\cir)$.
		For every $f\in \mathcal{P}_n(\cir)$, by Lemma \ref{lem:moi_rep}, we have the representations
		\begin{align}\label{new1}
			\mathcal{R}_{n}^{\mult}(U_0,A,f)
			=\dfrac{1}{(n-1)!}\int_{0}^{1}(1-t)^{ n-1}\,\dfrac{d^n}{ds^n}\Big|_{s=t}f(U_s)\,dt,		\end{align}
		%where the integral converges in the trace class norm by Theorem \ref{derivativeformulaunitarypath}. %and Definition \ref{moi}
		%By continuity of the trace,
 and
 \begin{align}\label{a1}
			\nonumber&\Tr\left(\mathcal{R}_{n}^{\mult}(U_0,A,f) \right)\\			=&\,\dfrac{(i)^n}{(n-1)!}\int_{0}^{1}(1-t)^{ n-1}\,\sum_{k=1}^{n-1}\,\sum_{\substack{j_1+\cdots+j_k=n\\j_1,\ldots,j_k\geq 1}}\,\dfrac{n!}{j_1!\cdots j_k!}\Tr\Big(T_{f^{[k]}}^{ U_t,\ldots,U_t}( A^{j_1}U_t,\ldots,A^{j_k}U_t)\Big)dt.
		\end{align}
		For every $k=1,\dots,n-1$, consider the linear functional $T_k$ on $L^2(\cir,\dm)$ given by
		\begin{align}
			\label{aa}
			\begin{cases}
				T_k((z^l)^{(k)})
				=\dfrac{(i)^n}{(n-1)!}\int_0^1(1-t)^{ n-1}\!\sum\limits_{\substack{j_1+\cdots+j_k=n\\j_1,\ldots,j_k\geq 1}}\,\dfrac{n!}{j_1!\cdots j_k!}\Tr\Big(T_{(z^l)^{[k]}}^{U_t,\ldots, U_t}(A^{j_1}U_t,\ldots,A^{j_k}U_t)\Big)dt\\
				\hspace*{4in}\text{ for } l=k,\dots,n-1, \\%[10pt]
				T_k(z^q)=0 \hspace*{3.25in}\text{ for } q\in\mathbb{Z}\setminus\{0,\ldots,n-k-1\}.
			\end{cases}
		\end{align}
		Let $p(z)$ be a polynomial of degree at most $n-1$. By Corollary \ref{useofcyclicity} and linearity of $T_k$,
		\begin{align}
			\label{aa1}
			|T_k(p^{(k)})|\leq {\tilde{d}}_{k,n}\|p^{(k)}\|_{L^\infty(\cir)}\|A\|_n^n.
		\end{align}
		Consider the trigonometric polynomial $\tilde{p}(\theta)=p^{(k)}(e^{i\theta})$ of degree at most $n-k-1$.
		By Jackson's inequality \cite[Ch.~5, (3.1.1), p.~495]{MMR},
		$\|\tilde p\|_{L^\infty[0,2\pi]}\le
		2\sqrt{n-k-1}\Big(\int_0^{2\pi}|p(\theta)|^2\,d\theta\Big)^{\frac12}$.
		Hence,
		\begin{align}
			\label{aa2}
			\|p^{(k)}\|_{L^\infty(\cir)}\le
			2\sqrt{2\pi}\sqrt{n-k-1}\,\|p^{(k)}\|_{L^2(\cir,\dm)}.
		\end{align}
		Combining \eqref{aa}, \eqref{aa1}, \eqref{aa2} implies
		\begin{align}
			\label{aa3}
			\|T_k\|\leq d_{k,n}\|A\|_n^n.
		\end{align}
		
		Therefore, by the Riesz representation theorem for the functionals in $(L^2(\cir,\dm))^*$,  it follows from \eqref{a1}, \eqref{aa}, \eqref{aa3} that there exist functions $\tilde{\eta}_k \in L^1(\cir,\dm)$, $k=1,\dots,n-1$, satisfying		
		\[\|{\tilde{\eta}}_k\|_1\leq d_{k,n}\|A\|^n_n\] and
		\begin{align}\label{aaa}
			\int_{\cir}z^l\, \overline{z}\,{\tilde{\eta}}_k(z)\,dz=0
		\end{align}
		for all $l\in\mathbb{Z}\setminus\{0,\ldots, n-k-1\}$ and such that \eqref{a3} holds for every $f\in\mathcal{P}_n(\cir)$, where
		\begin{align*}
			\eta_k(z)=\frac{1}{2\pi i}\,\bar{z}\,\tilde{\eta}_k(z),\quad k=1,\dots,n-1.
		\end{align*}
		It follows from \eqref{aaa} that
		\begin{align*}
			\tilde\eta_k\in\{\dots,z^{-n},\dots,z^{-n+k},z,\dots,z^n,\dots\}^\perp={\rm span}\{1,\dots,\overline{z}^{\,n-k-1}\}.
		\end{align*}
		
Since $\Fpt{n}\subset\Fcal_{n-1}(\T)$ (see Proposition \ref{new_class}\eqref{new_class_ii}), we obtain $\Fpt{n}=\mathcal{P}_n(\cir)+\tFpt{n}$. Given $f\in\Fpt{n}$, let $p\in\Fpt{n}$ and $g\in\tFpt{n}$ be such that $f=p+g$.
		It follows from the equations \eqref{aaa} and \eqref{a3} that
		\begin{align}\label{a4}
			\Tr\left(\mathcal{R}_{n}^{\mult}(U_0,A,p) \right) =\sum_{k=1}^{n-1}\int_{\cir}\,p^{(k)}(z)\,\eta_k(z)\,dz
			=\sum_{k=1}^{n-1}\int_{\cir}\,f^{(k)}(z)\,\eta_k(z)\,dz.
		\end{align}
		By the linearity of the trace, from the equations \eqref{trfor} and \eqref{a4}, we conclude that
		\begin{align*}
			\Tr\left(\mathcal{R}_{n}^{\mult}(U_0,A,f) \right)
			=&\Tr\left(\mathcal{R}_{n}^{\mult}(U_0,A,p) \right)+\Tr\left(\mathcal{R}_{n}^{\mult}(U_0,A,g) \right)\\
			=&\sum_{k=1}^{n-1}\int_{\cir}f^{(k)}(z)\eta_k(z)\,dz + \int_{\cir}g^{(n)}(z)\eta_n(z)\,dz\\
			=&\sum_{k=1}^{n-1}\int_{\cir}f^{(k)}(z)\eta_k(z)\,dz + \int_{\cir}f^{(n)}(z)\eta_n(z)\,dz \hspace*{.5in}(\,\text{since } f^{(n)}=g^{(n)}),
			%=&\sum_{k=1}^{n}\int_{\cir}f^{(k)}(z)\eta_k(z)\,dz,
		\end{align*}
		completing the proof of \eqref{trformulaunigen}.
	\end{proof}
	
	Next, we extend the trace formula \eqref{trformulaunigen} to pairs of contractions by utilizing the result of Theorem \ref{extuniatry}, Sch\"affer unitary matrix dilation, and the following observation. %furnishes a crucial tool for establishing our main results.
	
	\begin{lma}\label{traingulartrace}
		Let $\hil_i, i=1,2,3$, be separable Hilbert spaces. Let $T\in\mathcal{B}_1(\hil_1\oplus\hil_2\oplus\hil_3)$ have the block matrix representation of the form
\begin{align*}
T=\begin{blockarray}{ccccc}
\hil_1 & \hil_2 & \hil_3 & \\[4pt]
\begin{block}{[ccc]cc}
P_{\hil_1}TP_{\hil_1}&  P_{\hil_1}TP_{\hil_2}&  P_{\hil_1}TP_{\hil_3} & \hil_1 \\[3pt]
P_{\hil_2}TP_{\hil_1}& P_{\hil_2}TP_{\hil_2}& P_{\hil_2}TP_{\hil_3} & \hil_2\\[3pt]
P_{\hil_3}TP_{\hil_1}&P_{\hil_3}TP_{\hil_2}& P_{\hil_3}TP_{\hil_3}& \hil_3\\[3pt]
\end{block}
\end{blockarray},
\end{align*}
where $P_{\hil_1}, P_{\hil_2}, P_{\hil_3}$ are orthogonal projections of $\hil_1\oplus\hil_2\oplus\hil_3$ onto $\hil_1\oplus 0\oplus 0$, $0\oplus \hil_2\oplus 0$, and $0\oplus 0\oplus \hil_3$, respectively. Then
		\begin{align}\label{blocktraceid}
			\Tr(T)= \Tr \left( P_{\hil_1}TP_{\hil_1}\right)+\Tr \left( P_{\hil_2}TP_{\hil_2}\right)+ \Tr \left(P_{\hil_3}TP_{\hil_3}\right).
		\end{align}
	\end{lma}

	\begin{proof}
		Since
		%\begin{align}\label{blocktraceid1}
		$\Tr(T)= \sum_{i,j=1}^{3}\Tr \left(	P_{\hil_i}TP_{\hil_j}\right)$,
		%\end{align}
		\eqref{blocktraceid} follows from the pairwise orthogonality of $P_{\hil_i}$ and cyclicity of the trace.
	\end{proof}

Let $f\in\Fpt{n}$ and define
%be such that $\phi(e^{it})=\sum\limits_{k=-\infty}^{\infty}\hat{\phi}(k)e^{ikt}$.
$$f_{+}(e^{it}):=\sum\limits_{k=0}^{\infty}\hat{f}(k)e^{ikt}\, \text{ and }\; f_{-}(e^{it}):=\sum\limits_{k=1}^{\infty}\hat{f}(-k)e^{ikt},$$
where the series converge absolutely by  Proposition \ref{new_class}\eqref{new_class_ii}. Note that $f(e^{it})=f_{+}(e^{it})+f_{-}(e^{-it})$.
For a contraction $T$ on $\hil$, we set
	\begin{align}\label{INTeq1}
		f_{+}(T):=\sum\limits_{k=0}^{\infty}\hat{f}(k)T^k,\quad   f_{-}(T^*):=\sum\limits_{k=1}^{\infty}\hat{f}(-k){T^*}^k,\quad \text{and}\quad f(T):=f_{+}(T)+f_{-}(T^*).
	\end{align}
	The functions of contractions given by \eqref{INTeq1} were initially considered in \cite{Neid_1988}.
	
\smallskip
	
\begin{Not}\label{not_mult_cont}
		Let $n\in\Nat$, $n\geq 2$. Let $T_0$ be a contraction and $B$ a bounded self-adjoint operator on $\hil$. Define $$T_s=e^{isB}T_0,\quad s\in [0,1].$$ Given $f\in\Fpt{n}$, define $f(T_s)$ by \eqref{INTeq1} and set
		\begin{align}\label{taylorremmult}			
			\mathcal{R}_{n}^{\mult}(T_0,B,f)
			:=f(T_1)-f(T_0)-\sum_{k=1}^{n-1}\dfrac{1}{k!}\dfrac{d^k}{ds^k}\Big|_{s=0}f(T_s).
		\end{align}
	\end{Not}

\begin{thm}\label{extcontracmulti}
		Assume Notations \ref{not_mult_cont} and assume that $B\in\mathcal{B}_n(\hil)$.
		Then, $\mathcal{R}_{n}^{\mult}(T_0,B,f)\in \mathcal{B}_1(\hil)$ for every $f\in\Fpt{n}$. Furthermore, there exists a constant $d_{n}>0$ and functions $\eta_1,\ldots\eta_n\in  L^1(\cir,\dm)$ satisfying
		\begin{align}
			\label{etakbound}
			\|\eta_k\|_1\leq d_{n}\|B\|_n^n,\quad k=1,\dots,n,
		\end{align}
		such that \begin{align}\label{trformulauni2}
			\Tr\left(\mathcal{R}_{n}^{\mult}(T_0,B,f) \right)=\sum_{k=1}^{n}\int_{\cir}f^{(k)}(z)\eta_k(z)\,dz,
		\end{align}
		for every $f\in\Fpt{n}$. Moreover, if $\hat f(k)=0$ for $k=1,\dots,n-1$, then \eqref{trformulauni2} holds with $\eta_k=0$ for $k=1,\dots,n-1$.
	\end{thm}
	
	\begin{proof}
		%Let $\ell_2(\hil)=\oplus_{1}^{\infty}\hil$.
Let $U_0:=U_{T_0}$ be Sch\"affer's unitary matrix dilation on $\ell_2(\hil)\oplus \hil\oplus \ell_2(\hil)$ of $T_0$ whose block matrix representation is given by \eqref{schafferdil}.
		%\begin{align}\label{schafferdil}
		%	U_{T_0}=\begin{blockarray}{ccccc}
		%		\ell_2(\hil) & \hil & \ell_2(\hil) &  \\[4pt]
		%		\begin{block}{[ccc]cc}
		%			S^*& 0& 0 & \ell_2(\hil) \\[3pt]
		%			D_{T_0^*}P_\hil& T_0& 0 & \hil\\[3pt]
		%			-T_0^*P_\hil& D_{T_0}& S& \ell_2(\hil)\\[3pt]
		%		\end{block}
		%	\end{blockarray},
		%\end{align}
		%where $S$ is the unilateral shift on $\ell_2(\hil)$ of multiplicity $\dim(\hil)$ and $P_\hil$ is the orthogonal projection of $\ell_2(\hil)$ onto $\hil\oplus 0\oplus0\oplus\cdots$.
		Let $U_s$ be the unitary dilation on $\ell_2(\hil)\oplus \hil\oplus \ell_2(\hil)$ of $T_s$ whose block representation is given by
		\begin{align}
			U_s=e^{isA}U_0=\begin{bmatrix}
				I& 0& 0 \\
				0& e^{isB}& 0 \\
				0& 0& I
			\end{bmatrix}\begin{bmatrix}
				S^*& 0& 0 \\
				D_{T_0^*}P_\hil& T_0& 0 \\
				-T_0^*P_\hil& D_{T_0}& S
			\end{bmatrix}=\begin{bmatrix}
				S^*& 0& 0 \\
				e^{isB}D_{T_0^*}P_\hil& T_s& 0 \\
				-T_0^*P_\hil& D_{T_0}& S
			\end{bmatrix},
		\end{align}
		where $A=\begin{bmatrix}
			0& 0& 0 \\
			0& B& 0 \\
			0& 0& 0
		\end{bmatrix}: \ell_2(\hil)\oplus \hil\oplus \ell_2(\hil)\to \ell_2(\hil)\oplus \hil\oplus \ell_2(\hil)$ is the self-adjoint extension of $B$.
		Note that $A\in\mathcal{B}_n(\ell_2(\hil)\oplus \hil\oplus \ell_2(\hil))$ and $U_0$  satisfy the hypothesis of Theorem \ref{extuniatry} and that
		\begin{align}
			\label{AnBn}
			\|A\|_n=\|B\|_n.
		\end{align}
		
		Let $ k,l\in\Nat$. By a straightforward computation similar to the one in \eqref{uniderfor},
		\begin{align*}%\label{new3}
			&\frac{d^l}{ds^l}\bigg|_{s=0}(U_s^*)^k=\left(\frac{d^l}{ds^l}\bigg|_{s=0}U_s^k\right)^*,\\
			\nonumber
			&\frac{d}{ds}\Big|_{s=t}(T_s^*)^k=\lim_{h\to 0}\frac{(T_{t+h}^*)^k-(T_{t}^*)^k}{h}=\,\left(\lim_{h\to 0}\frac{T_{t+h}^k-T_{t}^k}{h}\right)^*
			=\left(\frac{d}{ds}\bigg|_{s=t}T_s^k\right)^*.
		\end{align*}
		Therefore, it follows from the above identities and \eqref{remua} that
		\begin{align}\label{adjoitrelation}
			\mathcal{R}_{n}^{\textit {Mult}}(U_0,A,z^{-k})=\, {U_1^*}^k-{U_0^*}^k-\sum_{l=1}^{n-1}\dfrac{1}{l!}\dfrac{d^l}{ds^l}\Big|_{s=0}{U_s^*}^k=
			\left(\mathcal{R}_{n}^{\textit {Mult}}(U_0,A,z^{k})\right)^*.
		\end{align}
		Similarly,
		\begin{align}\label{adjoitrelation1}
			\mathcal{R}_{n}^{\textit {Mult}}(T_0,B,z^{-k})=\left(\mathcal{R}_{n}^{\textit {Mult}}(T_0,B,z^{k})\right)^*.
		\end{align}
		Following the computation of \cite[Theorem 5.3.4]{Skbook} for unitary operators, we obtain
		\begin{align}\label{cont_der}
			\dfrac{d^l}{dt^l}\Big|_{t=s}T_{t}^k
			=\sum_{r=1}^{l}~\sum_{\substack{l_1+l_2+\cdots+l_{r}=l\\l_1,l_2,\ldots,l_{r}\geq 1}}~\dfrac{l!}{l_1!\cdots l_r!}\Bigg[\sum_{\substack{\alpha_0+\alpha_1+\cdots+\alpha_r=k\\\alpha_0\geq 0; \,\alpha_1,\ldots,\alpha_r\geq 1}}T_{s}^{\alpha_0}(iB)^{l_1}T_{s}^{\alpha_1}\cdots (iB)^{l_{r}}T_{s}^{\alpha_r}\Bigg]
		\end{align}
		for $l,k\in\Nat$. Recall that we also have a similar formula for the power of the unitary operator $U_s^k$ (see \eqref{uniderfor}).
		Consequently, there exists $c_l>0$ such that
		\begin{align}\label{der_est}
			\Big\|\dfrac{d^l}{dt^l}\Big|_{t=s}T_t^k \Big\|\leq \,c_l\,k^{l}\,\|B\|^l\; \text{ and }\; \Big\|\dfrac{d^l}{dt^l}\Big|_{t=s}U_t^k \Big\|\leq \,c_l\,k^{l}\,\|A\|^l
		\end{align} for $l,k\in\Nat$. We also have the representation
		\[ \dfrac{d^l}{dt^l}\Big|_{t=s_2}T_{t}^k-\dfrac{d^l}{dt^l}\Big|_{t=s_1}T_{t}^k
		=\int_{s_1}^{s_2}\dfrac{d^{l+1}}{dt^{l+1}}\Big|_{t=s}\, T_t^k\,ds,\]
for every $l\in\Nat\cup\{0\}$ and $k\in\Nat$, which along with \eqref{der_est} implies \[\left\|\dfrac{d^l}{dt^l}\Big|_{t=s_2}T_{t}^k-\dfrac{d^l}{dt^l}\Big|_{t=s_1}T_{t}^k \right\|\leq |s_2-s_1|\,c_{l+1}|k|^{l+1}\|B\|^{l+1}.\]  A completely analogous bound holds for $T^*_t$ in place of $T_t$.

Let $f\in\Fpt{n}$. Since $\sum_{k\in\mathbb{Z}}|k|^{n-1}|\hat{f}(k)|<\infty$, the above estimates imply that
\begin{align*}%\label{der_iden}
\dfrac{d^l}{dt^l}\Big|_{t=0}f(T_t)
=\sum_{k=0}^\infty\hat{f}(k)\,\dfrac{d^l}{dt^l}\Big|_{t=0}T_t^k
+\sum_{k=1}^\infty\hat{f}(-k)\,\dfrac{d^l}{dt^l}\Big|_{t=0}(T_t^*)^k,\quad
 l=1,\dots,n-1,
\end{align*}
where the series converge absolutely in the operator norm. Hence,
		\begin{align*}
			&\mathcal{R}_{n}^{\textit{Mult}}(T_0,B,f)
			=f(T_1)-f(T_0)-\sum_{l=1}^{n-1}\dfrac{1}{l!}\dfrac{d^l}{ds^l}\Big|_{s=0}
			\Big(\sum_{k=0}^{\infty}\hat{f}(k)T_s^k+\sum_{k=1}^{\infty}\hat{f}(-k)(T_s^*)^k\Big)\\
			&=\sum_{k=-\infty}^{\infty}\hat{f}(k)T_1^k-\sum_{k=-\infty}^{\infty}\hat{f}(k)T_0^k-\sum_{l=1}^{n-1}\dfrac{1}{l!}
			\Big(\sum_{k=0}^{\infty}\hat{f}(k)\dfrac{d^l}{ds^l}\Big|_{s=0}T_s^k
			+\sum_{k=1}^{\infty}\hat{f}(-k)\dfrac{d^l}{ds^l}\Big|_{s=0}(T_s^*)^k\Big).
		\end{align*}
Since all the series on the right-hand side converge absolutely in the operator norm, we can rearrange the terms and obtain
\begin{equation}\label{new2}
\begin{split}
&\mathcal{R}_{n}^{\textit {Mult}}(T_0,B,f)= \sum_{k=-\infty}^{\infty}\hat{f}( k)\mathcal{R}_{n}^{\textit {Mult}}(T_0,B,z^k)\;\text{ and, similarly,}\\
&\mathcal{R}_{n}^{\textit {Mult}}(U_0,A,f)= \sum_{k=-\infty}^{\infty}\hat{f}( k)\mathcal{R}_{n}^{\textit {Mult}}(U_0,A,z^k).
\end{split}
\end{equation}	
%where the series converge absolutely in the operator norm.
		
By examining the block matrix representations of $U_1^k-U_0^k$ and $\dfrac{d^l}{ds^l}\Big|_{s=0}U_s^k$, below we confirm the block representation for the remainder
		\begin{align}\label{blockrep}
			\mathcal{R}_{n}^{\textit {Mult}}(U_0,A,z^k)=\begin{blockarray}{ccccc}
				\ell_2(\hil) & \hil & \ell_2(\hil) &  \\[4pt]
				\begin{block}{[ccc]cc}
					0& 0& 0 & \ell_2(\hil) \\[3pt]
					\#& \mathcal{R}_{n}^{\mult}(T_0,B,z^k)& 0 & \hil\\[3pt]
					\#& \#& 0& \ell_2(\hil)\\[3pt]
				\end{block}
			\end{blockarray}
		\end{align} for every $k\in\Nat$, where `$\#$' denotes a non-zero entry of a matrix. Indeed,
		\begin{align}
			\label{expression1}
			\nonumber
			&U_1^k-U_0^k=\sum_{j=0}^{k-1}U_1^j(U_1-U_0)U_0^{k-1-j}\\
			\nonumber
			&=\sum_{j=0}^{k-1}\begin{bmatrix}
				S^*& 0& 0 \\
				e^{iB}D_{T_0^*}P_\hil& T_1& 0 \\
				-T_0^*P_\hil& D_{T_0}& S
			\end{bmatrix}^j
			\begin{bmatrix}
				0& 0& 0 \\
				(e^{isB}-I)D_{T_0^*}P_\hil& T_1-T_0& 0 \\
				0& 0& 0
			\end{bmatrix}
			\begin{bmatrix}
				S^*& 0& 0 \\
				D_{T_0^*}P_\hil& T_0& 0 \\
				-T_0^*P_\hil& D_{T_0}& S
			\end{bmatrix}^{k-1-j}\\
			\nonumber
			&=\sum_{j=0}^{k-1}\begin{bmatrix}
				{S^*}^j& 0& 0 \\
				\#& T_1^j& 0 \\
				\#& \#& S^j
			\end{bmatrix}
			\begin{bmatrix}
				0& 0& 0 \\
				(e^{isB}-I)D_{T_0^*}P_\hil& T_1-T_0& 0 \\
				0& 0& 0
			\end{bmatrix}
			\begin{bmatrix}
				{S^*}^{k-j-1}& 0& 0 \\
				\#& T_0^{k-j-1}& 0 \\
				\#& \#& S^{k-j-1}
			\end{bmatrix}\\
			&=\begin{bmatrix}
				0& 0& 0 \\
				\#& T_1^k-T_0^k& 0 \\
				\#& \#& 0
			\end{bmatrix}.
		\end{align}
		By Theorem \ref{derivativeformulaunitarypath},
		\begin{align}
			\label{expression2}
			\nonumber
			\dfrac{d^l}{ds^l}\Big|_{s=0}U_s^k=&	\sum_{r=1}^{l}~\sum_{\substack{j_1+j_2+\cdots+j_{r}=l\\j_1,j_2,\ldots,j_{r}\geq 1}}~\dfrac{l!}{j_1!\cdots j_r!}\Bigg[\sum_{\substack{\alpha_0+\alpha_1+\cdots+\alpha_r=k\\\alpha_0\geq 0;\,\alpha_1,\ldots,\alpha_r\geq 1}}U_0^{\alpha_0}(iA)^{j_1}U_0^{\alpha_1}\cdots (iA)^{j_{r}}U_0^{\alpha_r}\Bigg]\\
			\nonumber				=&\sum_{r=1}^{l}~\sum_{\substack{j_1+j_2+\cdots+j_{r}=l\\j_1,j_2,\ldots,j_{r}\geq 1}}~\dfrac{l!}{j_1!\cdots j_r!}
			\begin{bmatrix}
				{S^*}^{\alpha_0}& 0& 0 \\
				\#& T_0^{\alpha_0}& 0 \\
				\#& \#& S^{\alpha_0}
			\end{bmatrix}\begin{bmatrix}
				0& 0& 0 \\
				0& (iB)^{j_1}& 0 \\
				0& 0& 0
			\end{bmatrix}\times\cdots\\
			\nonumber
			&\hspace{4cm}\times\begin{bmatrix}
				0& 0& 0 \\
				0& (iB)^{j_r}& 0 \\
				0& 0& 0
			\end{bmatrix}
			\begin{bmatrix}
				{S^*}^{\alpha_r}& 0& 0 \\
				\#& T_0^{\alpha_0}& 0 \\
				\#& \#& S^{\alpha_r}
			\end{bmatrix}\\
			=&
			\begin{bmatrix}
				0& 0& 0 \\
				\#& \dfrac{d^l}{ds^l}\Big|_{s=0}T_s^k& 0 \\
				\#& \#& 0
			\end{bmatrix},
		\end{align}
		where the last equality follows from \eqref{cont_der}.
		Thus, combining \eqref{expression1} and \eqref{expression2} yields \eqref{blockrep}.
		
		The properties \eqref{adjoitrelation}, \eqref{adjoitrelation1}, and \eqref{blockrep} imply
		\begin{align}\label{compress}
			\mathcal{R}_{n}^{\mult}(T_0,B,z^k)=Q_\hil\, \mathcal{R}_{n}^{\mult}(U_0,A,z^k)\big|_\hil,\quad k\in\mathbb{Z},
		\end{align}
		where $Q_\hil$ is the orthogonal projection of $\ell_2(\hil) \oplus \hil \oplus \ell_2(\hil)$ onto the subspace $0 \oplus \hil \oplus 0$.
		Combining \eqref{new2} and \eqref{compress} gives
		\begin{align*}
			\mathcal{R}_{n}^{\mult}(T_0,B,f) = Q_\hil \, \mathcal{R}_{n}^{\mult}(U_0,A,f) \big|_\hil, \quad f \in \Fpt{n}.
		\end{align*}
 Hence, by Theorem \ref{extuniatry}, $\mathcal{R}_{n}^{\mult}(T_0,B,f)\in\boh$ for every $f\in\Fpt{n}$.

It follows from \eqref{blockrep} and \eqref{adjoitrelation} that the block matrix representation of $\mathcal{R}_{n}^{\mult}(U_0,A,z^k)$ is given by \begin{align*}
		\mathcal{R}_{n}^{\mult}(U_0,A,z^k) = \begin{bmatrix}
			0 & \# & \# \\
			\# & \mathcal{R}_{n}^{\mult}(T_0,B,z^k) & \# \\
			\# & \# & 0
		\end{bmatrix},\quad k\in\Z.
\end{align*}
Hence the block matrix representation of $\mathcal{R}_{n}^{\mult}(U_0,A,f)$ is given by
\begin{align}\label{newblockrep}
				\mathcal{R}_{n}^{\mult}(U_0,A,f) = \begin{bmatrix}
					0 & \# & \# \\
					\# & \mathcal{R}_{n}^{\mult}(T_0,B,f) & \# \\
					\# & \# & 0
				\end{bmatrix}.
\end{align}
By Lemma \ref{traingulartrace}, it follows from  \eqref{newblockrep} that, for every $f\in\Fpt{n}$,
		\begin{align}\label{finaltraceidentity}
			\Tr\left(\mathcal{R}_{n}^{\mult}(T_0,B,f)\right)
			=\Tr\left(\mathcal{R}_{n}^{\mult}(U_0,A,f)\right).
		\end{align}
Applying \eqref{trformulaunigen} of Theorem \ref{extuniatry} on the right hand side of \eqref{finaltraceidentity} completes the proof of \eqref{trformulauni2}. By Theorem \ref{extuniatry} and \eqref{AnBn}, we obtain that the respective functions $\eta_k$ satisfy \eqref{etakbound}.
\end{proof}
	
In Theorem \ref{Helt1} below we extend the first order trace formula of \cite{AC_KBS_JOT} to the higher order case. Our main tools are the higher order trace formula for contractions over $\cir$ derived in Theorem~\ref{extcontracmulti} and the Poisson integral extension of a function form $\cir$ to $\mathbb{D}$.
	
	\medskip
	
	\paragraph{\bf Poisson integral}
	We recall that the Poisson integral $Pf$ of $f\in L^1(\mathbb{T})$ is defined by
	\begin{align}\label{poissonint}
		(Pf)(z) =  \frac{1}{2\pi}\int_0^{2\pi}\frac{1-|z|^2}{\left|\textup{e}^{it}-z\right|^2}
		~f(\textup{e}^{it})\,dt,\quad z\in\D.
	\end{align}
	By noting that
	\begin{align*}%\label{poissonintz}
		\frac{1-|z|^2}{\left|\textup{e}^{it}-z\right|^2}
		= 1 + \sum_{n=1}^{\infty} \bar{z}^n\textup{e}^{int}+\sum_{n=1}^{\infty} z^n\textup{e}^{-int},\quad z\in\D,
	\end{align*}
	we obtain from \eqref{poissonint} that
	\begin{align}\label{poissonintzz}
		(Pf)(z) =  \hat{f}(0) + \sum_{n=1}^{\infty} \hat{f}(-n)\bar{z}^n+\sum_{n=1}^{\infty} \hat{f}(n) z^n,\quad  z\in \mathbb{D}.
	\end{align}
	Let $f \in L^1(\cir)$ and define the extension $\widetilde{f}$ of $f$ from $\mathbb{T}$ to $\overline{\mathbb{D}}=\mathbb{D}\cup \mathbb{T}$ by
\begin{equation}\label{newrepresentation}
		\widetilde{f}(z,\bar{z}):= (Pf)(z)
		= \hat{f}(0) + \sum\limits_{n=1}^{\infty} \hat{f}(-n)\bar{z}^n+\sum\limits_{n=1}^{\infty} \hat{f}(n) z^n,\quad z\in\overline{\mathbb{D}}.
\end{equation}
%Let $n\in\Nat$.
%By Proposition \ref{new_class}\eqref{new_class_i}, $\Fcal_n(\cir)\subset\Fpt{n}$.
%For $f\in\Fcal_n(\cir)$ and a contraction $T$ on $\hil$, define
%	\begin{align}\label{newoperatorfunct}
%		\widetilde{f}(T,T^*):=\hat{f}(0)I + \sum\limits_{n=1}^{\infty} \hat{f}(-n)\,T^{*^n}+\sum\limits_{n=1}^{\infty}\hat{f}(n)\,T^{^n}.
%	\end{align}
%	Note that the series in \eqref{newoperatorfunct} converge absolutely by the property of $\Fcal_n(\cir)$.
	
%	\begin{Not}\label{not_helton}
%		Let $n\in\Nat, n\geq 2$. Let $T_0$ be a contraction and $B$ be a bounded 	self-adjoint operator, and $$T_s=e^{isB}T_0,\quad s\in [0,1].$$ For $f\in \Fcal_n(\cir)$ and $\widetilde{f}$ given by \eqref{newrepresentation}, consider the $n$th Taylor remainder
%		\begin{align*}
			%\label{RT0T0*}
%			\mathcal{R}_{n}^{\mult}(T_0, T_0^*,B,\widetilde{f}\,):=\widetilde{f}(T_1,T_1^*)-\widetilde{f}(T_0,T_0^*)
%			-\sum_{k=1}^{n-1}\dfrac{1}{k!}\dfrac{d^k}{ds^k}\Big|_{s=0}\widetilde{f}(T_s,T_s^*),
%		\end{align*}
%		where $\widetilde{f}(T_s,T_s^*)$ is given by \eqref{newoperatorfunct}.
%	\end{Not}
	\begin{thm} \label{Helt1}
		Assume Notations \ref{not_mult_cont} and assume that $B\in\mathcal{B}_n(\hil)$.
		Then, $\mathcal{R}_{n}^{\mult}(T_0,B,f)\in\mathcal{B}_1(\hil)$ for $f\in\Fcal_n(\cir)$ and
		\begin{align}\label{helt1}
			\textup{Tr} \left(\mathcal{R}_{n}^{\mult}(T_0,B,f)\right)=
			\sum_{k=1}^{n}\; \lim\limits_{R\,\uparrow 1} \int\limits_{\{z:\;|z|\leq R<1\}}%J\big(\widetilde{\eta_k},\widetilde{f^{(k-1)}}\big)(z,\bar{z})
			\left(\frac{\partial \widetilde{\eta_k}}{\partial z}\frac{\partial  \widetilde{f^{( k-1)}}}{\partial \bar{z}} -\frac{\partial  \widetilde{f^{(k-1)}}}{\partial z}\frac{\partial  \widetilde{\eta_k}}{\partial \bar{z}}\right)
			\,dz\wedge d\bar{z},
		\end{align}
		where $\widetilde{\eta_k}, \widetilde{f^{(k-1)}}$ are given by \eqref{newrepresentation}, the functions $\eta_k$, $k=1,\dots,n$, are the spectral shift functions for the pair $(T_0, B)$ provided by Theorem \ref{extcontracmulti},
and $dz\wedge d\bar{z}$ is the Lebesgue measure on $\mathbb{D}$.
	\end{thm}
	
	\begin{proof}
		Let $f\in \Fcal_n(\cir)$.
		%Note that $\mathcal{R}_{n}^{\mult}(T_0, T_0^*,B,\widetilde{f})$ equals $\mathcal{R}_{n}^{\mult}(T_0,B,f)$ defined in \eqref{taylorremmult}.
		 Then, by Theorem \ref{extcontracmulti}, $\mathcal{R}_{n}^{\mult}(T_0, B,f)\in\mathcal{B}_1(\hil)$ and there exist $\eta_1,\ldots,\eta_{n}$ in $L^{1}(\cir)$ such that
		\begin{align}\label{rep1}
			\Tr \left(\mathcal{R}_{n}^{\mult}(T_0, B,f)\right)=
			\sum_{k=1}^{n}\int_{\cir} f^{(k)}(z)\eta_k(z)\,dz.
		\end{align}
		Since $f\in\Fcal_n(\cir)$, we obtain $f^{(k)}\in L^1(\cir)$ for $k=1,\dots,n$, and
		\begin{align}\label{rep2}
			\nonumber\int_{\cir} f^{(k)}(z)\eta_k(z)~dz=&\int_{\cir} \left(\sum_{l=-\infty}^{\infty} \widehat{f^{(k)}}(l)z^l\right)\eta_k(z)\,dz\\
			\nonumber=&2\pi i\sum_{l=-\infty}^{\infty} \widehat{f^{(k)}}(l)\widehat{\eta}_k(-(l+1))\\
			=&2 \pi i \sum_{l=-\infty}^{\infty}(l+1) \widehat{f^{(k-1)}}(l+1)\widehat{\eta}_k(-(l+1)).
		\end{align}
		The final equality arises from the relationship $\widehat{f^{(k)}}(l) = (l+1)\widehat{f^{(k-1)}}(l+1)$. By employing a computation akin to that in the proof of \cite[(3.11)]{AC_KBS_JOT}, we derive the representation for the integral
		\begin{align}\label{rep3}
			\lim_{R\,\uparrow 1}\int_{\{z:\;|z|\leq R<1\}} \left(\frac{\partial \widetilde{\eta_k}}{\partial z}\frac{\partial  \widetilde{f^{( k-1)}}}{\partial \bar{z}} -\frac{\partial  \widetilde{f^{(k-1)}}}{\partial z}\frac{\partial  \widetilde{\eta_k}}{\partial \bar{z}}\right)\,dz\wedge d\bar{z} = 2\pi i \sum_{l=-\infty}^{\infty}\,l\,\widehat{f^{(k-1)}}(l)\,\hat{\eta_k}(-l)
		\end{align}
		for each $k=1,\dots,n$. Combining the equations \eqref{rep1}, \eqref{rep2}, and \eqref{rep3} completes the proof of \eqref{helt1}.
	\end{proof}
	
	\section{ Higher order trace formulas for maximal dissipative operators}\label{resolventcomtr}
	
	In this section we derive higher order trace formulas for a pair of maximal dissipative operators  $L_0,L_1$ without imposing the stringent assumptions of \cite[Theorem 5.3]{Ch_Pr_NYJM} that $\dim \ker(L_j+iI) = \dim \ker(L_j^*-iI)$ and $\Im L_j=\frac{1}{2i}(L_j-L_j^*)\in\mathcal{B}_{n/2}(\hil)$, $j=0,1$. Our result builds upon the trace formula for multiplicative paths of contractions established in Theorem \ref{extcontracmulti}.
	
	%\begin{dfn}\end{dfn}
	We recall that a densely defined linear operator $L$ (possibly unbounded) in $\hil$ is called dissipative if $\Im\la L\,\xi, \xi\ra\le0$ for all $\xi\in\text{Dom}(L)$.  A dissipative operator is called maximal dissipative if it does not have a proper dissipative extension.
	
	The Cayley transform of the maximal dissipative operator $L$ is defined by
	\begin{align}\label{cont_dis_cayley}
		T=(L+iI)(L-iI)^{-1}.
	\end{align}
	It is well known that $T$ is a contraction. Moreover, a contraction $T$ is the Cayley transform of a maximal dissipative operator $L$ if and only if $1$ is not an eigenvalue of $T$ (see, e.g., \cite[Theorem 4.1]{NFbook}). In the latter case, the inverse Cayley transform of $T$ is given by
	\begin{align}\label{cay_cont_dis}
		L=i(T+I)(T-I)^{-1}.
	\end{align}
	For $\psi(\lambda)=\sum_{k\in\mathbb{Z}}a_k\left(\frac{\lambda+i}{\lambda-i}\right)^k\in \FptR{n}$, we define
	
	\begin{align*}
		\psi(L):=&\sum_{k=0}^{\infty} a_k\big((L+iI)(L-iI)^{-1}\big)^k+\sum_{k=1}^{\infty} a_{-k}{\big((L+iI)(L-iI)^{-1}\big)^*}^k\\
		=&\sum_{k=0}^{\infty}a_k T^k+\sum_{k=1}^{\infty}a_{-k} {T^*}^k.
	\end{align*}
Below we use the notation shortcut $\frac{L+iI}{L-iI}:=(L+iI)(L-iI)^{-1}$.
%Let
%\begin{align*}
%\tFptR{n}:=&\left\{\psi(\lambda)=f\left(\frac{\lambda+i}{\lambda-i}\right):\; f\in\cblue %\tFpt{n}\right\}.
%\end{align*}
	
\begin{Not}\label{not_dis}
		Let $n \in \Nat$, $n \geq 2$. Let $L_0$ be a maximal dissipative operator in $\hil$ and $T_0$ be its Cayley transform given by \eqref{cont_dis_cayley}.  Let $B = B^* \in \mathcal{B}_n(\hil)$, set $T_1 := e^{iB} T_0$, and let $L_1$ be the inverse Cayley transform of $T_1$ given by \eqref{cay_cont_dis}. Assume that $1$ is not an eigenvalue of $T_1$.
	\end{Not}
	
	%Example of such a pair of operators $(L_0, L_1)$ satisfying Hypothesis in the above theorem includes a pair of self-adjoint operators $(H_0, H_1)$ having a resolvent difference in $\bnh$. Additionally, one can consider a strict contraction $T_0$ (i.e., $|T_0| < 1$), and then take its Cayley transformation, say $L_0$. In this case, $L_0, L_1$ satisfies hypothesis in the above theorem.
	
	\begin{thm}\label{Trace formula: self-adjoint multiplicative path}
		Assume Notations \ref{not_dis}. Then, there exist a constant $c_n>0$ and $\gamma_k(\lambda):=(\lambda-i)^{-2}\,\eta_{k}\left(\frac{\lambda+i}{\lambda-i}\right)$ with $\eta_k$ as in Theorem \ref{extuniatry}, $k=1,\dots,n$, satisfying
		\[\|\eta_k\|_1\leq c_n \|(L_1-iI)^{-1}-(L_0-iI)^{-1}\|_n^n\] and
		\begin{align}	
			\label{dissipativetrf}
			\nonumber	&\Tr\Bigg(\psi(L_1)-\psi(L_0)-\sum_{k\in\mathbb{Z}}a_k\sum_{l=1}^{n-1}\sum_{r=1}^{l}~\sum_{\substack{l_1+l_2+\cdots+l_{r}=l\\l_1,l_2,\ldots,l_{r}\geq 1}}~\dfrac{1}{l_1!\cdots l_r!}\\
			\nonumber
			&\times \Bigg[\sum_{\substack{\alpha_0+\alpha_1+\cdots+\alpha_r=k\\\alpha_0\geq 0; \,\alpha_1,\ldots,\alpha_r\geq 1}}\left(\frac{L_0+iI}{L_0-iI}\right)^{\alpha_0}(iB)^{l_1}
			\left(\frac{L_0+iI}{L_0-iI}\right)^{\alpha_1}\cdots (iB)^{l_{r}}\left(\frac{L_0+iI}{L_0-iI}\right)^{\alpha_r}\Bigg]\Bigg)\\
			=&\sum_{k=1}^{n}\dfrac{i^{k-1}}{2^{k-1}}\int_\R (\lambda-i)^k\,\dfrac{d^{k-1}}{d\lambda^{k-1}}
			\left((\lambda-i)^k\psi'(\lambda)\right)\gamma_k(\lambda)\,d\lambda
		\end{align}
		for every $\psi(\lambda)=\sum_{k\in\mathbb{Z}}
		a_k\left(\frac{\lambda+i}{\lambda-i}\right)^k\in\FptR{n}$.
		Moreover, if $a_k=0$, $k=1,\dots,n-1$, then \eqref{dissipativetrf} holds with $\gamma_k=0$ for $k=1,\dots,n-1$.
	\end{thm}
	
\begin{proof}
The result follows upon subsequently applying \eqref{taylorremmult}, \eqref{cont_der}, \eqref{cont_dis_cayley} on the left-hand side of \eqref{trformulauni2} and changing the variable on the right-hand side of \eqref{trformulauni2} as outlined in the proof of \cite[Theorem 3.5]{PoSkSu15}.
\end{proof}
	
	\section{ Simplified higher order trace formulas for unitaries and resolvent comparable self-adjoints}\label{simpletrformula}
	As evident from Theorems \ref{derivativeformulaunitarypath}, \ref{extuniatry}, and \ref{extcontracmulti}, the left-hand sides of the equations \eqref{trformulaunigen} and \eqref{trformulauni2} exhibit a highly intricate structure. In Theorem \ref{linunitarytraceformula} below we derive an alternative trace formula for unitary operators that does not involve computation of the operator derivatives along multiplicative paths.  As a consequence of the latter result, we obtain higher order trace formulas for resolvent comparable self-adjoint perturbations, considerably simplifying the trace formulas in \cite[Theorem 5.3]{Skadv2017}.
	
	\begin{Not}\label{not_lin_unitary}
		Let $n\in\Nat$, $n\geq 2$. Let $U_0,U_1$ be unitary operators on $\hil$ and let $$U_s=U_0+s(U_1-U_0),\quad s\in[0,1].$$ For $f\in\Fpt{n}$, consider the modified Taylor remainder
		\begin{align*}%\label{remsimple}
			\mathcal{R}_{n}^{\lin}(U_1,U_0,f):= f(U_1)-f(U_0)-\sum_{k=1}^{n-1}T_{f^{[k]}}^{U_0,\ldots,U_0}(U_1-U_0,\ldots,U_1-U_0).
		\end{align*}
	\end{Not}
	
	%For example, if $f(z)=z^{-n}$, for some $n\in\Nat$,  $$ \dfrac{d^1}{ds^1}\bigg|_{s=0}\,f(U_0+s(U_1-U_0))=\left(T_{{z^n}^{[1]}}^{U_0,U_0}(U_1-U_0).\right)^*\neq T_{(\bar{z}^n)^{[1]}}^{U_0,U_0}(U_1-U_0)$$
	
	\begin{thm}\label{linunitarytraceformula}
		Assume Notations \ref{not_lin_unitary} and assume $U_1-U_0\in\bnh$. Then,  $\mathcal{R}_{n}^{\lin}(U_1,U_0,f)\in \mathcal{B}_1(\hil)$ for every $f\in\Fpt{n}$. Furthermore, there exist a constant $c_{n}>0$ and a function $\eta_n$  in $ L^1(\cir,\dm)$, unique up to a polynomial of degree at most $n-1$, satisfying
		\begin{align}\label{shiftest}
			\|\eta_n\|_1\leq c_{n}\|U_1-U_0\|_n^n
		\end{align} such that
		\begin{align}\label{trformulauni1}
			\Tr\left(\mathcal{R}_{n}^{\lin}(U_1,U_0,f)\right)=\int_{\cir}f^{(n)}(z)\eta_n(z)\,dz
		\end{align}
		for every $f\in\Fpt{n}$.
	\end{thm}
	
	\begin{proof}
		Let $f\in\Fpt{n}$. A repeated application of Lemma \ref{algebraicprop} yields
		\begin{align}
			\label{53}
			\nonumber
			\mathcal{R}_{n}^{\lin}(U_1,U_0,f)
			&=T_{f^{[1]}}^{U_0,U_1}(U_1-U_0)
			-\sum_{k=1}^{n-1}T_{f^{[k]}}^{U_0,\ldots,U_0}(U_1-U_0,\ldots,U_1-U_0)\\
			&=T_{f^{[n]}}^{U_0,U_1,U_0,\ldots,U_0}(U_1-U_0,\ldots,U_1-U_0).	
		\end{align}
		
		It follows from Definition \ref{moi} that $T_{f^{[n]}}^{ U_0,U_1,U_0,\ldots,U_0}(U_1-U_0,\ldots,U_1-U_0)\in\boh$ and, hence, $\mathcal{R}_{n}^{\lin}(U_1,U_0,f)\in\boh$. By Corollary \ref{useofcyclicity}, we have
		\begin{align*}
			\left|\Tr\left(\mathcal{R}_{n}^{\lin}(U_1,U_0,f)\right)\right|\leq \,c_n\,\|f^{(n)}\|_\infty\,\|U_1-U_0\|_n^n.
		\end{align*}
		Hence, there exists a measure $\mu_{n,U_0,U_1}$ satisfying
		\begin{align}\label{measureest}
			\|\mu_{n,U_0,U_1}\|\leq c_n\,\|U_1-U_0\|_n^n
		\end{align} and
		\begin{align}\label{lintracfor}
			\Tr\left(\mathcal{R}_{n}^{\lin}(U_1,U_0,f)\right)=\int_{\cir} f^{(n)}(z)\,d\mu_{n,U_0,U_1}.
		\end{align}
		Next we prove that the measure $\mu_{n,U_0,U_1}$ is absolutely continuous with respect to the Lebesgue measure on the circle. Since $U_1U_0^*$ is a unitary operator, there is a self-adjoint operator $A$ with the spectrum $\sigma(A)$ in $(-\pi,\pi]$ such that  $U_1U_0^*=e^{iA}$, and hence $$U_1=e^{iA}U_0.$$ Let $\{e_i\}_{i\in\Nat}$ be any orthonormal basis of $\hil$ and $E(\cdot)$ the spectral measure of $A$. Then, it follows from the spectral theorem and the inequality $|x|\leq~\dfrac{\pi}{2}|e^{ix}-1|$ for $x\in{(-\pi,\pi]}$ that
		\begin{align*}
			\|A\|_n^n =\sum_{i=1}^{\infty}\la |A|^ne_i,e_i\ra
			=\sum_{i=1}^{\infty}\int_{-\pi}^{\pi}|\lambda|^n\, \|E(d\lambda) e_i\|^2
			& \leq~\left(\dfrac{\pi}{2}\right)^n\sum_{i=1}^{\infty}\int_{-\pi}^{\pi}|e^{i\lambda}-1|^n~\|E(d\lambda) e_i\|^2\\
			& = \left(\dfrac{\pi}{2}\right)^n\,\|U_1-U_0\|_n^n.
		\end{align*}
		Since $U_1-U_0\in\bnh$, we conclude that $A\in\bnh$.
		
		\smallskip
		\noindent{\it Case 1:} $n=2$.
		
		We have
		\begin{align*}
			\mathcal{R}_{2}^{\lin}(U_1,U_0,f)
			=\mathcal{R}_{2}^{\mult}(U_0,A,f)-T_{f^{[1]}}^{U_0,U_0}\left(\left(e^{iA}-I-iA\right)U_0\right),
		\end{align*}
		where $\mathcal{R}_{2}^{\mult}(U_0,A,f)$ is defined in \eqref{remua}.
		Therefore, by Theorem \ref{extuniatry} and Corollary \ref{useofcyclicity}, there are $\eta_{12}$, $\eta_{22}$ in $L^1(\mathbb{T})$, and a measure $\nu$ on $\cir$ such that
		\begin{align*}
			\|\nu\|\le\|e^{iA}-I-iA\|_1
		\end{align*}
		and
		\begin{align*}
			\Tr(\mathcal{R}_{2}^{\lin}(U_1,U_0,f))= \int_{\mathbb{T}} f''(z)\eta_{22}(z)\,dz+\int_{\mathbb{T}} f'(z)\eta_{12}(z)\,dz-\int_{\mathbb{T}} f'(z)d\nu.
		\end{align*}
		Integrating by parts yields the existence of an integrable function $\eta_2$ such that \eqref{trformulauni1} holds (see, e.g., the proof of \cite[Theorem 3.2]{PoSkSu}).
		\smallskip
		
		\noindent{\it Case 2:} $n\geq 3$.

		Note that
		\begin{align}\label{breakintwoparts}		\mathcal{R}_{n}^{\lin}(U_1,U_0,f)=\mathcal{R}_{n-1}^{\lin}(U_1,U_0,f)-T_{f^{[n-1]}}^{U_0,\ldots,U_0}(U_1-U_0,\ldots,U_1-U_0).	
		\end{align}
		
		Let $\{A_m\}_{m\in \Nat}\subset\mathcal{B}_1(\hil)$ be a sequence of  self-adjoint operators converging to $A$ in the norm $\|\cdot\|_n$ and such that
		$$\sup_{m\in\Nat}\|A_m\|\le \|A\|.$$
		Define
		$$U_{0m}=e^{iA_m}U_0.$$ It follows from Corollary \ref{useofcyclicity} and  the equations \eqref{53} and \eqref{breakintwoparts} that there exist measures $\mu_{n,m}$, $ m\in\Nat$, satisfying
		\begin{align*}
			\|\mu_{n,m}\|\le c_n\|U_1-U_0\|_n^n
		\end{align*}
		and
		\begin{align}\label{hgorder}
			\Tr\left(\mathcal{R}_{n}^{\lin}(U_{0m},U_0,f)\right)=\int_{\cir} f^{(n-1)}(z)\,d\mu_{n,m}.
		\end{align}
		Integrating \eqref{hgorder} by parts yields the existence of a sequence $\{\eta_{n,m}\}_{m\in\Nat}\subset L^1(\mathbb{T})$ such that
		\begin{align}\label{allorder}
			\Tr\left(\mathcal{R}_{n}^{\lin}(U_{0m},U_0,f)\right)=\int_{\cir} f^{(n)}(z)\eta_{n,m}(z)\,dz.
		\end{align}
By Lemma \ref{explip0},
		\begin{align}
			\label{explip}
			\|e^{i A_m}-e^{iA}\|_n
			%\leq \sum_{k=1}^{\infty}\sum_{ p=0}^{k-1}\frac{\|(iA)^p(iA-iA_m)(iA_m)^{k-1-p}\|_n}{k!}
			\leq e^{\|A\|}\,\|A-A_m\|_n.
		\end{align}
		It follows from \eqref{explip} that $U_{0m}$ converges to $U_1$ in $\mathcal{B}_n(\hil)$. Hence, $$V_m:=U_{0m}-U_0$$ forms a Cauchy sequence in $\mathcal{B}_n(\hil)$. In particular, there exists $M>0$ such that
		%, and consequently, it is uniformly bounded.
		$$\sup_{m\in\Nat}\|V_m\|_n\leq M.$$
		
		\noindent By \eqref{53} and telescoping,
		\begin{align*}
			&\mathcal{R}_{n}^{\lin}(U_{0m},U_0,f)
			-\mathcal{R}_{n}^{\lin}(U_{0p},U_0,f)\\
			&= T_{f^{[n-1]}}^{U_{0},U_{0m},U_0,\ldots,U_0}(V_m,\ldots,V_m)-T_{f^{[n-1]}}^{ U_{0},U_{0p},U_0,\ldots,U_0}(V_p,\ldots,V_p)\\
			&\hspace*{1.5in}-\left(T_{f^{[n-1]}}^{U_0,\ldots,U_0}(V_m,\ldots,V_m)-T_{f^{[n-1]}}^{U_0,\ldots,U_0}(V_p,\ldots,V_p)\right)\\
			&=T_{f^{[n-1]}}^{U_0,U_{0m},U_{0p},U_0\ldots,U_0}(V_m,U_{0m}-U_{0p},\ldots,V_m)\\
			&\hspace*{1.5in}+\sum_{k=1}^{n-1}\Big[T_{f^{[n-1]}}^{U_0,U_{0p},U_0\ldots,U_0}(V_p,\ldots,V_p,\underbrace{V_m-V_p}_{k-\text{th entry}},V_m,\ldots,V_m)\\
			&\hspace*{2.5in}-T_{f^{[n-1]}}^{U_0,U_0,U_0\ldots,U_0}(V_p,\ldots,V_p,\underbrace{V_m-V_p}_{k-\text{th entry}},V_m,\ldots,V_m)\Big]\\
			&=T_{f^{[n]}}^{U_0,U_{0m},U_{0p},U_0\ldots,U_0}(V_m,U_{0m}-U_{0p},\ldots,V_m)+T_{f^{[n]}}^{U_0,U_{0p},U_0\ldots,U_0}(V_m-V_p,U_{0p}-U_0,V_m,\ldots,V_m)\\
			&\hspace{1.5in}+\sum_{k=3}^{n}T_{f^{[n]}}^{U_0,U_{0p},U_0\ldots,U_0}(V_p,U_{0p}-U_0,V_p\ldots,V_p,\underbrace{V_m-V_p}_{k-\text{th entry}},V_m,\ldots,V_m).
		\end{align*}
		The latter along with Corollary \ref{useofcyclicity} implies
		\begin{align*}			
			&\left|\Tr\left(\mathcal{R}_{n}^{\lin}(U_{0m},U_0,f)\right)
			-\Tr\left(\mathcal{R}_{n}^{\lin}(U_{0p},U_0,f)\right)\right|\\
			&\leq c_n M^{n-1}\|f^{(n)}\|_\infty\max\big\{\|U_{0m}-U_{0p}\|_n, \|U_0-U_{0p}\|_n,\|V_m-V_p\|_n\big\},
		\end{align*}
		which further implies
		\begin{align}
			\label{finalconv}\nonumber&\sup_{f\in\Fcal_n(\cir):\;\|f^{(n)}\|_\infty\leq 1}\left|\Tr\left(\mathcal{R}_{n}^{\lin}(U_{0m},U_0,f)\right)
			-\Tr\left(\mathcal{R}_{n}^{\lin}(U_{0p},U_0,f)\right)\right|\\
			&\leq c_n M^{n-1}\max\big\{\|U_{0m}-U_{0p}\|_n, \|U_0-U_{0p}\|_n,\|V_m-V_p\|_n\big\}\longrightarrow 0 \text{ as } m,p\to \infty.
		\end{align}
		
		Applying \cite[Lemma 4.3]{Skadv2017} along with \eqref{allorder} and \eqref{finalconv} yields
		\begin{align*}
			\|[\eta_{n,m}]-[\eta_{n,p}]\|_{L^1(\cir)/\mathcal{P}_n}\longrightarrow 0 \text{ as } m,p\to \infty.
		\end{align*}
		%\Skco{$\eta_m$ and $\eta_p$ have not been introduced yet.}
		Hence, there exists an $L^1(\mathbb{T})$-function $\eta_n$, unique up to an additive polynomial of degree $n-1$, satisfying \eqref{trformulauni1}. %\eqref{lintracfor}
		The estimate \eqref{shiftest} follows from \eqref{measureest} and \eqref{lintracfor}. %for each $n\in\Nat$. This completes the proof.
	\end{proof}
	
	We note that the trace formula \eqref{trformulauni1} does not translate from  unitaries to contractions via the Sch\"affer matrix dilation because the block matrix representation of $T_{f^{[k]}}^{U_{T_0},\ldots,U_{T_0}}(U_{T_1}-U_{T_0},\ldots,U_{T_1}-U_{T_0})$ contains both upper and lower triangular matrices.
	
	\bigskip
	
	Below we apply Theorem \ref{linunitarytraceformula} to derive trace formulas for resolvent comparable self-adjoint operators.	
	
	\begin{Not}\label{uniselfnot}
		Let $n \in \Nat$, $n \geq 2$. Let $H_0$ be a closed, densely defined self-adjoint operator in $\hil$ and $V$ a self-adjoint operator in $\mathcal{B}(\hil)$. Denote $H_1 := H_0 + V$. Define the unitary operators via the Cayley transforms of $H_0$ and $H_1$:
		\[U_0 := (H_0 + iI)(H_0 - iI)^{-1}\, \text{ and }\; U_1 := (H_1 + iI)(H_1 - iI)^{-1}.\]
	\end{Not}
	
	Observe that
	\begin{align*}
		U_1 - U_0 &= -2i(H_1 - iI)^{-1}V(H_0 - iI)^{-1},\\
		U_1U_0^* &= I - 2i(H_1 - iI)^{-1}V(H_0 + iI)^{-1}.
	\end{align*}
	If $(H_1 - iI)^{-1} - (H_0 - iI)^{-1}\in\bnh$, then $U_1-U_0,A\in\bnh$, where $A = A^*$ is such that its spectrum is contained in $(-\pi, \pi]$ and $U_1 U_0^* = e^{iA}$ (see the proof of Theorem \ref{linunitarytraceformula}).
	
	\begin{thm}\label{Trace formula: self-adjoint linear path}
		Assume Notations \ref{uniselfnot} and assume that \[(H_1-iI)^{-1}-(H_0-iI)^{-1}\in\bnh.\]
		Then, there exist $c_n>0$ and $\gamma_n(\lambda):=(\lambda-i)^{-2}\,\eta_{n}\left(\frac{\lambda+i}{\lambda-i}\right)$, with $\eta_n$ as in Theorem \ref{linunitarytraceformula}, satisfying
		\[\|\gamma_n\|_1\leq c_n \|(H_1-iI)^{-1}-(H_0-iI)^{-1}\|_n^n\] and
		\begin{align}\label{simpleresolventtraceformula}
			\nonumber&\Tr\left(\psi(H_1)-\psi(H_0)-\sum_{k=1}^{n-1}\sum_{1\leq j_1<\cdots<j_k\leq n-1} T_{\psi^{[k]}}^{H_0,H_0,\ldots,H_0}(V_{j_1},\ldots,V_{j_k})\right)\\
			&=\dfrac{i^{n-1}}{2^{n-1}}\int_\R (\lambda-i)^ n\,\dfrac{d^{n-1}}{d\lambda^{n-1}}\left((\lambda-i)^n\psi'(\lambda)\right)\gamma_n(\lambda)\,d\lambda
		\end{align}
		for every $\psi\in\FptR{n}$,
		where
		\begin{align*}
			V_{j_l}=\big((I-V(H_1-iI))^{-1}V(H_0-iI)^{-1}\big)^{j_l-j_{l-1}}(I-V(H_1-iI))^{-1}V
		\end{align*}
		and $j_0:=0$.
	\end{thm}
	
	\begin{proof}
		The result follows upon applying the substitution of \cite[Theorem 5.2]{Skadv2017} on the left-hand side of \eqref{trformulauni1} and changing the variable on the right-hand side of \eqref{trformulauni1} as outlined in the proof of \cite[Theorem 3.5]{PoSkSu15}.
	\end{proof}
	
	\smallskip
	
	\noindent\textit{Acknowledgment}: The authors thank the anonymous referee for several suggestions that helped to improve the exposition of the paper. A. Chattopadhyay is supported by the Core Research Grant (CRG), File No: CRG/2023/004826, of SERB. A. Skripka is supported in part by Simons Foundation Grant MP-TSM-00002648. C. Pradhan acknowledges support from the IoE post-doctoral fellowship at IISc Bangaloreas, as well as the NBHM post-doctoral fellowship (File No. 0204/27/(9)/2023/R\&D-II/11882).

\end{document}